\documentclass[10pt]{amsart}
\usepackage{a4wide}
\usepackage{amsmath, amssymb, amsfonts,enumerate}
\usepackage[all]{xy}
\usepackage{amscd}
\usepackage{comment}
\usepackage{mathtools}
\usepackage{tikz} 
\usepackage{tikz-cd} 

 \newtheorem{theorem}{Theorem}[section]
\newtheorem{lemma}[theorem]{Lemma}
\newtheorem{corollary}[theorem]{Corollary}
\newtheorem{proposition}[theorem]{Proposition}

 \theoremstyle{definition}
 \newtheorem{definition}[theorem]{Definition}
 \newtheorem{remark}[theorem]{Remark}

\numberwithin{equation}{section}
\newcommand {\N}{\mathbb{N}} 
\newcommand {\Z}{\mathbb{Z}} 

\newcommand{\BB}{\mathcal{B}}

\DeclareMathOperator{\spa}{span}
\DeclareMathOperator{\Per}{Per}
   \DeclareMathOperator{\GL}{GL}

\DeclareMathOperator{\LCA}{LCA}
\DeclareMathOperator{\Ker}{Ker}
\DeclareMathOperator{\End}{End}

\DeclareMathOperator{\Id}{Id}

\DeclareMathOperator{\Mat}{Mat}

\DeclareMathOperator{\NW}{NW}
\DeclareMathOperator{\Rec}{R}
\DeclareMathOperator{\Crec}{CR}

\begin{document}
\title{On linear shifts of finite type and their endomorphisms}  
\author[T. Ceccherini-Silberstein]{Tullio Ceccherini-Silberstein}
\address{Dipartimento di Ingegneria, Universit\`a del Sannio, C.so
Garibaldi 107, 82100 Benevento, Italy}
\email{tullio.cs@sbai.uniroma1.it}
\author[M. Coornaert]{Michel Coornaert}
\address{Universit\'e de Strasbourg, CNRS, IRMA UMR 7501, F-67000 Strasbourg, France}
\email{michel.coornaert@math.unistra.fr}
\author[X.K. Phung]{Xuan Kien Phung}
\address{D\'epartement de math\'ematiques, Universit\'e du Qu\'ebec \`a Montr\'eal, 
Case postale 8888, succursale centre-ville, Montr\'eal (Qu\'ebec) H3C 3P8, Canada}
\email{phungxuankien1@gmail.com}
\begin{abstract}
Let $G$ be a group and let $A$ be a finite-dimensional vector space over an arbitrary field $K$. 
We study finiteness properties of linear subshifts $\Sigma \subset A^G$ and the dynamical behavior of
linear cellular automata $\tau \colon \Sigma \to \Sigma$.
We say that $G$ is of $K$-linear Markov type if, for every finite-dimensional vector space $A$ over $K$,
all linear subshifts $\Sigma \subset A^G$ are of finite type. 
We show that $G$ is of $K$-linear Markov type if and only if the group algebra $K[G]$ is one-sided Noetherian.
We prove that a linear cellular automaton $\tau$ is nilpotent if and only if its limit set, i.e., 
the intersection of the images of its iterates, reduces to the zero configuration.  
If $G$ is infinite, finitely generated, and $\Sigma$ is topologically mixing,
we show that $\tau$ is nilpotent if and only if its limit set is finite-dimensional. 
A new characterization of the limit set of $\tau$ in terms of pre-injectivity is also obtained. 
\end{abstract}
\date{\today}
\maketitle

\setcounter{tocdepth}{1}
\tableofcontents
\section{Introduction} 
Let $G$ be a group and let $A$ be a set, called the \emph{alphabet}.
The set $A^G  \coloneqq   \{x  \colon G \to A\}$, consisting of all maps from $G$ to $A$,  
is called the set of \emph{configurations} over the group $G$ and the alphabet $A$.
We equip $A^G = \prod_{g \in G} A$ with its \emph{prodiscrete uniform  structure}, i.e., 
the product uniform structure obtained by taking the discrete uniform structure on each factor $A$ of $A^G$.
Thus, two configurations are ``close'' if they coincide on a ``large'' subset of $G$.
Note that $A^G$ is a totally disconnected Hausdorff space and that $A^G$ is compact if and only if $A$ is finite.
The \emph{shift action} of the group $G$ on $A^G$ 
is the action defined by $(g,x) \mapsto g x$, where $gx(h) \coloneqq x(g^{-1}h)$ for all 
$g,h \in G$ and $x \in A^G$.
This action is uniformly continuous with respect to the prodiscrete uniform structure. 
\par
A closed $G$-invariant subset $\Sigma \subset A^G$ is called a \emph{subshift} of $A^G$.
\par
Given subsets $D \subset G$ and $P \subset A^D$, the set 
\begin{equation}
\label{e:sft} 
\Sigma(D,P) = \Sigma(A^G;D,P) \coloneqq \{x \in A^G: (g^{-1}  x)\vert_{D} \in P \text{ for all } g \in G\}
\end{equation}
is a $G$-invariant subset of $A^G$ (here $(g^{-1}x)\vert_D \in A^D$ denotes the restriction of the configuration $g^{-1}x$ to $D$).
When $D$ is finite, $\Sigma(D,P)$ is also closed in $A^G$, and therefore is a subshift. 
One then says that $\Sigma(D,P)$ is the \emph{subshift of finite type}, briefly \emph{SFT}, 
associated with $(D,P)$ and that $D$ (resp.\ $P$) is a \emph{defining memory set} 
(resp.\ a \emph{defining set of admissible patterns}) for $\Sigma$. 
Note that a defining set of admissible patterns for an SFT is not necessarily finite.
\par
Let $B$ be another alphabet set. 
A map $\tau \colon B^G  \to A^G$ is called a \emph{cellular automaton}, briefly a \emph{CA}, 
if there exist a finite subset $M \subset G$ and a map $\mu \colon B^M \to A$ such that 
\begin{equation} 
\label{e;local-property}
\tau(x)(g) = \mu((g^{-1}x)\vert_M)  \quad  \text{for all } x \in B^G \text{ and } g \in G.
\end{equation}
Such a set $M$ is then called a \emph{memory set} and $\mu$ is called a \emph{local defining map} for $\tau$.
It is immediate from the above definition that every CA $\tau \colon B^G \to A^G$ is uniformly continuous and $G$-equivariant 
(cf.~\cite[Theorem 1.1]{cc-TCS1}, see also \cite[Theorem 1.9.1]{book}). 
\par 
More generally, if $\Sigma_1 \subset B^G$ and $\Sigma_2 \subset A^G$ are subshifts, 
a map $\tau \colon \Sigma_1 \to \Sigma_2$ is called a \emph{CA} if it can be extended to a CA $\tilde{\tau} \colon B^G \to A^G$. 
\par 
Suppose now that $A$ and $B$ are vector spaces over a field $K$. 
Then $A^G$ and $B^G$ inherit a natural $K$-vector space structure. 
A subshift $\Sigma \subset A^G$ which is also a vector subspace of $A^G$ is called a \emph{linear subshift}. 
A $K$-linear CA $\tau \colon B^G \to A^G$ is called a \emph{linear CA}. 
Note that a CA $\tau \colon B^G \to A^G$ with memory set $M \subset G$ is linear if and only if the associated
local defining map $\mu \colon B^M \to A$ is $K$-linear (see \cite[Section 8.1]{book}).
\par
More generally, given linear subshifts $\Sigma_1 \subset B^G$ and $\Sigma_2 \subset A^G$, 
a map $\tau \colon \Sigma_1 \to \Sigma_2$ is called a \emph{linear CA}
if it is the restriction of some linear CA $\tilde{\tau} \colon B^G \to A^G$. 
\par
A linear subshift $\Sigma \subset A^G$ is called a \emph{linear-sofic subshift} provided there exists a vector space $B$, 
a linear SFT $\Sigma' \subset B^G$, and a linear CA $\tau\colon B^G \to A^G$ 
such that $\tau(\Sigma') = \Sigma$. 
\par
In our recent papers \cite{ccp-2020, phung-2020} we introduced the notion of an algebraic sofic subshift $\Sigma \subset A^G$, where $A$
is the set of $K$-points of an algebraic variety over an algebraically closed field $K$, and studied cellular 
automata $\tau \colon \Sigma \to \Sigma$ whose local defining maps are induced by algebraic morphisms. 
When referring to these notions, we shall refer to the ``algebraic setting''.
When the field $K$ is algebraically closed, linear-sofic subshifts and linear cellular automata are algebraic sofic subshifts
and algebraic cellular automata, respectively. 
Therefore, several results in \cite{ccp-2020, phung-2020} hold true in the present setting, even when the field $K$ is not algebraically closed, by a direct adaptation of the proofs given therein.
However the proofs in the algebraic setting are much more technical and involved, 
and one of the purposes of this paper is to present simpler and more direct proofs of these results in the
linear setting. We also obtain several new results and consequences as indicated below. 
\par
Our first result is a linear version of the well known characterization of SFT with finite alphabets by the descending chain condition 
(see \cite[Theorem 10.1]{ccp-2020} and \cite[Proposition~6.1]{phung-2020} for a similar result in the algebraic setting, and 
\cite[Proposition~9.17]{phung-2020} for the more general \emph{admissible group shifts}, \cite[Definition~9.11]{phung-2020}).

\begin{theorem}
\label{t:LSFT-DCC}
Let $G$ be a countable group and let $A$ be a finite-dimensional vector space over a field $K$. 
Let $\Sigma \subset A^G$ be a linear subshift. Then the following conditions are equivalent:
\begin{enumerate}[{\rm (a)}]
\item $\Sigma$ is a SFT;
\item every decreasing sequence of linear subshifts of $A^G$
\[
\Sigma_0 \supset \Sigma_1 \supset \cdots \supset \Sigma_n \supset \Sigma_{n+1} \supset \cdots
\]
such that $\Sigma = \bigcap_{n \in \N} \Sigma_n$, eventually stabilizes (that is, there exists
$n_0 \in \N$ such that $\Sigma_{n_0} = \Sigma_n$ for all $n \geq n_0$).
\end{enumerate}
\end{theorem}
 
\begin{corollary}
\label{c:DCC-FT}
Let $G$ be a countable group and let $A$ be a finite-dimensional vector space over a field $K$.
Then the following conditions are equivalent:
\begin{enumerate}[{\rm (a)}]
\item every linear subshift $\Sigma \subset A^G$ is an SFT;
\item $A^G$ satisfies the descending chain condition for linear subshifts, that is, every decreasing sequence of linear subshifts of $A^G$
\[
\Sigma_0 \supset \Sigma_1 \supset \cdots \supset \Sigma_n \supset \Sigma_{n+1} \supset \cdots
\]
eventually stabilizes.
\end{enumerate}
\end{corollary}

Given a field $K$, we say that a group $G$ is of \emph{$K$-linear Markov type} provided that the equivalent conditions in 
Corollary \ref{c:DCC-FT} hold for every finite-dimensional vector space $A$ over $K$. 

Let $G$ be a group and let $K$ be a field. Given $\alpha \in K[G]$ we define $\alpha^* \in K[G]$ by setting 
$\alpha^*(g) \coloneqq \alpha(g^{-1})$ for all $g \in G$. It is straightforward to check that
the map $\alpha \mapsto \alpha^*$ yields a $K$-algebra isomorphism of the group algebra $K[G]$ 
onto the opposite algebra $K[G]^{opp}$.  As a consequence, $K[G]$ is left-Noetherian
if and only if it is right-Noetherian and, if this is the case, we simply say that $K[G]$ is one-sided Noetherian. 
In the proofs, however, in order to use a working definition at hand, we shall always refer to left Noetherianity.

We have the following characterization of countable groups of $K$-linear Markov type.

\begin{theorem}
\label{t:noether}
Let $G$ be a countable group and let $K$ be a field.
Then the group algebra $K[G]$ is one-sided Noetherian if and only if $G$ is of $K$-linear Markov type.
\end{theorem}

\par
Recall that a group $G$ is said to be \emph{polycyclic} if it admits a subnormal series with cyclic factors, that is, 
a finite sequence $G  = G_0 \supset G_1 \supset \cdots  \supset G_n = \{1_G\}$ of subgroups such that $G_{i+1}$ is normal in $G_i$ and
$G_i/G_{i+1}$ is (possibly infinite) cyclic group, for $i=0,1,\ldots, n-1$.
More generally, $G$ is said to be \emph{polycyclic-by-finite} if it admits a polycyclic subgroup of finite index. 
\par
The following result is the linear version of a famous result by Klaus Schmidt \cite[Theorem 4.2]{schmidt-book}
(see also \cite{kitchens-schmidt}). 
The third-named author considered the notion of an \emph{admissible Artinian group structure} (cf.\ \cite[Definition~9.1]{phung-2020}):
this includes, for instance, a group structure of finite Morley rank, e.g.\ an algebraic group, an Artinian group, or an Artinian module.
Since a finite-dimensional vector space over a field $K$ is an Artinian $K$-module
and therefore it is naturally equipped with an admissible Artinian group structure (see \cite[Example 9.12]{phung-2020}),  
Corollary \ref{c:LMT} also constitutes a simpler case of the more general results 
\cite[Theorem~9.13]{phung-2020} or \cite[Theorem 1.11]{phung-2020}, where the alphabet set can be taken as an 
admissible Artinian group structure. 
Also note that the proof of \cite[Theorem 4.2]{schmidt-book} heavily relies on the fact that the alphabet set
therein is a \emph{compact} Lie group, so that the configuration space (equipped with the product topology) is itself compact.
We deduce Corollary~\ref{c:LMT} below from Theorem \ref{t:noether} and a result of Philip Hall 
(\cite{hall}, see also \cite[Corollary 2.8]{passman}) extending Hilbert's basis theorem on Noetherian rings.
An alternative self-contained proof of Corollary~\ref{c:LMT}, using only linear symbolic dynamics, 
is presented in Remark \ref{r:alternative}.

\begin{corollary}
\label{c:LMT} 
Let $K$ be a field. Then all polycyclic-by-finite groups (e.g., the free abelian groups $\Z^d$, $d \geq 1$) are of $K$-linear Markov type.
\end{corollary}

The free group $F_2$ (and, more generally, any group which contains a subgroup which is not finitely generated) is not of $K$-linear Markov type (see Section \ref{ss:GLMT}). Groups of $K$-linear Markov type, or, more generally, \emph{monoids of $K$-linear Markov type}, satisfy  interesting topological properties. For example, it is shown in~\cite{phung-shadow} that the natural action of
every finitely generated abelian monoid of linear CA on any linear subshift satisfies the shadowing property.\\

Let now $f \colon X \to X$ be a selfmap of a set $X$.
\par
One has $X \supset f(X) \supset f^2(X) \supset \cdots \supset f^n(X) \supset f^{n+1}(X) \supset \cdots$ and the
set $\Omega(f) \coloneqq \bigcap_{n \geq 1} f^n(X) \subset X$ is called the \emph{limit set} of $f$. 
This is the set of points of $X$ that occur after iterating $f$ arbitrarily many times. 
The notion of a limit set was introduced in the framework of cellular automata by Wolfram \cite{wolfram} and was subsequently investigated
for instance in \cite{culik-limit-sets-1989}, \cite{guillon-richard-2008}, \cite{kari}, \cite{milnor}, and \cite{ccp-2020}.
\par 
Observe that $f(\Omega(f)) \subset \Omega(f)$. The inclusion may be strict (cf.~\cite[Proposition A.2.(iii)]{ccp-2020} and
Example (3) in Section \ref{s:examples-nilp-linear})
and equality holds if and only if every $x \in \Omega(f)$ admits a \emph{backward orbit}, 
i.e., a sequence $(x_i)_{i \geq 0}$ of points of $X$ such that 
$x_0 = x$ and $f(x_{i + 1}) = x_i$ for all $i \geq 0$. 
Clearly, $f$ is surjective if and only if $\Omega(f) = X$. 
Note also that $\Per(f) \coloneqq \bigcup_{n \geq 1} \{x\in X: f^n(x)=x\}$, the set of $f$-\emph{periodic points}, is
contained in $\Omega(f)$ and that $\Omega(f^n) = \Omega(f)$ for every $n \geq 1$. 
One says that the map $f$ is \emph{stable} if $f^{n+1}(X)=f^n(X)$ for some $n \geq 1$. 
\par
Assume that $X$ is a topological space and $f \colon X \to X$ is a continuous map. 
One says that $x \in X$ is a \emph{recurrent} (resp.~\emph{non-wandering}) point of $f$ if 
for every neighborhood $U$ of $x$, there exists $n \geq 1$ such that $f^n(x) \in U$ 
(resp.~$f^n(U)$ meets $U$).
Let $\Rec(f)$ (resp.~$\NW(f)$) denote the set of recurrent (resp.~non-wandering) points of $f$.
It is immediate that $\Per(f) \subset \Rec(f) \subset \NW(f)$ and that $\NW(f)$ is a closed subset of $X$.
\par
Suppose now that $X$ is a uniform space and let $f \colon X \to X$ be a uniformly continuous map.
One says that a point $x \in X$ is \emph{chain-recurrent} if for every entourage $E$ of $X$ there exist an integer $n \geq 1$ and
a sequence of points $x_0,x_1,\dots,x_n \in X$ such that $x = x_0 = x_n$ and $(f(x_i),x_{i + 1}) \in E$ for all $0 \leq i \leq n - 1$.
We shall denote by $\Crec(f)$ the set of chain-recurrent  points of $f$.
Observe that $\Crec(f)$ is always closed in $X$.
\par
We shall establish the following result (compare with~\cite[Theorem 1.3]{ccp-2020} in the algebraic setting).

\begin{theorem}
\label{t:limit-set}
Let $G$ be a finitely generated group and let $A$ be a finite-dimensional vector space over a field $K$. 
Let $\Sigma \subset A^G$ be a linear subshift and let $\tau \colon \Sigma \to \Sigma$ be a linear CA.  
Then the following hold:
\begin{enumerate}[\rm (i)]
\item
$\Omega(\tau)$ is a linear subshift of $A^G$;
\item
$\tau(\Omega(\tau)) = \Omega(\tau)$; 
\item
$\Per(\tau) \subset \Rec(\tau) \subset \NW(\tau) \subset \Crec(\tau) \subset \Omega(\tau)$;
\item 
if $\Omega(\tau)$ is of finite type then $\tau$ is stable; 
\item 
if $\Omega(\tau)$ is finite-dimensional then $\tau$ is stable.  
\end{enumerate}
\end{theorem} 

In the above theorem, 
we may relax the condition on $G$ being finitely generated provided
we assume in addition that the linear subshift $\Sigma \subset A^G$ is linear-sofic. 
We thus have the following.

\begin{corollary}
\label{c:limit-set}
Let $G$ be a group and let $A$ be a finite-dimensional vector space over a field $K$. 
Let $\Sigma \subset A^G$ be a linear-sofic subshift (e.g., a linear SFT) 
and let $\tau \colon \Sigma \to \Sigma$ be a linear CA.  
Then properties (i) -- (v) in Theorem \ref{t:limit-set} hold.
\end{corollary}

\par 
As in \cite{ccp-2020}, the proof relies on the analysis of the so called \emph{space-time inverse system} 
associated with a CA (cf.~Section~\ref{s:space-time-system}). 
\par
Let $G$ be a group and let $A$ be a set. 
A CA $\tau \colon \Sigma_1 \to \Sigma_2$ between subshifts of $A^G$
is called \emph{pre-injective} if whenever $x,y \in A^G$ are two configurations that coincide outside of a finite subset of $G$ 
and satisfy $\tau(x)=\tau(y)$, then one has $x=y$. When $A$ is a vector space over a field $K$ and $\Sigma_1, \Sigma_2 \subset A^G$
are linear subshifts, a linear CA $\tau \colon \Sigma_1 \to \Sigma_2$ is pre-injective if and only if the
restriction of $\tau$ to the vector subspace of configurations in $\Sigma_1$ with finite support is injective.
A subshift $\Sigma \subset A^G$ is called \emph{strongly irreducible} 
if there exists a finite subset $\Delta \subset G$ such that for all $x, y \in A^G$ and for all finite subsets $E, F \subset G$ 
such that $E \cap F \Delta = \varnothing$, then there exists $z \in \Sigma$ such that $z\vert_E = x\vert_E$ and $z\vert_F = y\vert_F$.
\par
We obtain the following  characterization of limit sets of linear cellular automata in terms of pre-injectivity.

\begin{corollary}
\label{c:other-characterization}
Let $G$ be a polycyclic-by-finite group and let $A$ be a finite-dimensional vector space. 
Let $\Sigma \subset A^G$ be a strongly irreducible linear subshift and let $\tau \colon \Sigma \to \Sigma$ be a linear CA. 
Then $\Omega(\tau)$ is the largest strongly irreducible linear subshift $\Lambda \subset A^G$ contained in $\Sigma$ 
such that $\tau(\Lambda) \subset \Lambda$ and $\tau\vert_{\Lambda}$ is pre-injective.
\end{corollary} 

\par
Given a set $X$, one says that a map $f \colon X \to X$ is \emph{nilpotent} if there exist a constant map 
$c \colon X \to X$ and an integer $n_0 \geq 1$ such that $f^{n_0} = c$. This implies $f^n = c$ for all $n \geq n_0$. 
Such  a constant map $c$ is then unique and we say that the unique point $x_0 \in X$ such that 
$c(x) = x_0$ for all $x \in X$ is the \emph{terminal point} of $f$.
The terminal point of a nilpotent map is its unique fixed point.
Observe that if $f \colon X \to X$ is nilpotent with terminal point $x_0$ then $\Omega(f) = \{x_0\}$ is a singleton.
The converse is not true in general (cf.\ \cite[Proposition A.2.(ii)]{ccp-2020} and Example (1) in Section \ref{s:examples-nilp-linear}). 
\par
For linear cellular automata we establish the following characterization of nilpotency.

\begin{theorem}
\label{t:char-nilpotent-finite-lca-one}
Let $G$ be a group and let $A$ be a finite-dimensional vector space over a field $K$. 
Let $\Sigma \subset A^G$ be a linear-sofic subshift (e.g., a linear SFT) and let 
$\tau \colon \Sigma \to \Sigma$ be a linear CA.  
Then the following conditions are equivalent: 
\begin{enumerate}[\rm (a)]
\item
$\tau$ is nilpotent;
\item 
$\Omega(\tau) = \{0\}$.
\end{enumerate} 
\end{theorem} 

The analog of Theorem~\ref{t:char-nilpotent-finite-lca-one} for classical cellular automata follows from 
\cite[Theorem~3.5]{culik-limit-sets-1989}. In the algebraic setting, this corresponds to \cite[Theorem 1.4]{ccp-2020}.
\par
Given a set $X$, one says that a map $f \colon X \to X$ is \emph{pointwise nilpotent} if there exist a point $x_0 \in X$ 
such that for each $x \in X$ there exists an integer $n_x \geq 1$ such that $f^n(x) = x_0$ for all $n \geq n_x$. 
Such a point $x_0$ is clearly unique and it is called the \emph{terminal point} of $f$. 
If $f$ is nilpotent then it is also pointwise nilpotent and the terminal points relative to the two notions of nilpotency coincide.
\par
Let $G$ act on a Hausdorff topological space $X$. One says that the action is \emph{topologically mixing} provided that given
two nonempty open subsets $U, V \subset X$ there exists a finite subset $F \subset G$ such that $U \cap gV \neq \varnothing$
for all $g \in G \setminus F$. If $A$ is a set, a subshift $\Sigma \subset A^G$ is said to be \emph{topologically mixing} if
the restriction to $\Sigma$ of the $G$-shift is topologically mixing.
\par
We obtain the following characterization of nilpotency for linear cellular automata over infinite groups. 

\begin{theorem}
\label{t:char-nilpotent-finite-lca}
Let $G$ be a finitely generated infinite group and let $A$ be a finite-dimensional vector space over a field $K$. 
Let $\Sigma \subset A^G$ be a topologically mixing linear subshift (e.g., $\Sigma = A^G$) and let $\tau \colon \Sigma \to \Sigma$ be a linear CA.  
Then the following conditions are equivalent: 
\begin{enumerate}[\rm (a)]
\item
$\tau$ is nilpotent;
\item 
$\tau$ is pointwise nilpotent; 
\item 
there exists $n_0 \in \N$ such that $\tau^{n_0}(\Sigma)$ is finite-dimensional;
\item
$\Omega(\tau)$ is finite-dimensional;
\item 
$\Omega(\tau) = \{0\}$.
\end{enumerate} 
\end{theorem} 

In the above theorem, we may relax the condition of being finitely generated on the infinite group $G$ provided
we assume that the subshift $\Sigma \subset A^G$ is, in addition, linear-sofic. We thus have the following.

\begin{corollary}
\label{c:char-nilpotent-finite-lca}
Let $G$ be an infinite group and let $A$ be a finite-dimensional vector space over a field $K$. 
Let $\Sigma \subset A^G$ be a topologically mixing linear-sofic subshift (e.g., $\Sigma = A^G$) 
and let $\tau \colon \Sigma \to \Sigma$ be a linear CA.  
Then conditions (a) -- (e) in Theorem \ref{t:char-nilpotent-finite-lca} are all equivalent.
\end{corollary} 

The analog of Theorem~\ref{t:char-nilpotent-finite-lca-one} for classical cellular automata follows from 
\cite[Theorem~3.5]{culik-limit-sets-1989} (see also \cite[Corollary 4]{guillon-richard-2008}). 
In the algebraic setting, this corresponds to \cite[Theorem 1.5]{ccp-2020}, 
but the equivalence of (a), (b), (c), and (e) with the point (d) is a new result.\\  

The paper is organized as follows. In Section \ref{s:preliminaries} we fix notation and establish some preliminary results.
In particular, we study linear SFTs and show that every finite-dimensional linear subshift $\Sigma \subset A^G$ is of finite type
if the group $G$ is finitely generated (Proposition \ref{p:FD-SFT}). 
We introduce the notion of a memory set for a linear-sofic subshift $\Sigma$ and,
given a cellular automaton $\tau \colon \Sigma \to \Sigma$, we review the properties of the restriction cellular automaton
$\tau_H \colon \Sigma_H \to \Sigma_H$, for any subgroup $H$ containing both a memory set for $\tau$ and a memory set for $\Sigma$.
As an application, we establish a relation between the limit sets $\Omega(\tau)$ and $\Omega(\tau_H)$ of $\tau$ and $\tau_H$, 
respectively, and deduce that $\tau$ is nilpotent if and only if $\tau_H$ is nilpotent (Lemma \ref{l:restriction-ls}).
In Section \ref{s:STIS} we review from \cite{ccp-2020} the notion of space-time inverse system, together with its inverse limit,
associated with a cellular automaton $\tau \colon \Sigma \to \Sigma$, where $\Sigma$ is a linear-sofic subshift in $A^G$, 
with $G$ a countable group and $A$ a finite-dimensional vector space over a field $K$. As an application, in the subsequent section
we prove the \emph{closed image property} for linear cellular automata: we show that essentially under the above assumptions,
$\tau(\Sigma)$ is closed in the prodiscrete topology in $A^G$ (Theorem \ref{t:closed-image}).
In Section \ref{s:proofs} we present the proofs of all results stated in the Introduction.
In Section \ref{ss:GLMT} we further investigate the class of groups of $K$-linear Markov type. We show
that this class is closed under the operation of taking subgroups, quotients, and extensions by finite and cyclic groups,
and that it is contained in the class of Noetherian groups (the latter are the groups satisfying the maximal condition on subgroups).
Finally, in the last section we present some examples/counterexamples and discuss some further remarks.
In particular, in Subsection~\ref{s:CIP-examples} we present an example of a linear cellular automaton $\tau \colon A^G \to A^G$, 
where $G$ is any non-periodic group (e.g., $G = \Z$) and $A$ is any infinite-dimensional vector space, which does not satisfy the closed image property.
At last, in Subsection~\ref{s:examples-nilp-linear} we study nilpotency and pointwise nilpotency for
linear cellular automata over infinite-dimensional vector spaces and present some examples of the associated limit sets.
As a byproduct, we show that the conclusions of Theorem \ref{t:char-nilpotent-finite-lca}
may fail to hold, in general, if the finite-dimensionality of the alphabet set $A$ is dropped from the assumptions.\\

\noindent
{\bf Acknowledgments.} 
We express our deepest gratitude to the anonymous referee for his thorough reading, for pointing out to us a few inaccuracies, and 
providing most valuable comments and remarks.

\section{Preliminaries}
\label{s:preliminaries}

\subsection{Notation}
We use the symbols $\Z$ for the integers and $\N$ for the non-negative integers.
\par
We write  $A^B$ for the set consisting of all maps from a set $B$ into a set $A$. 
Let $C \subset B$. If $x \in A^B$, we denote by $x\vert_C$ the restriction of $x$ to $C$, that is, the map $x\vert_C \colon C \to A$ given by $x\vert_C(c) = x(c)$ for all $c \in C$.
If $X \subset A^B$, we set $X_C \coloneqq \{x\vert_C \colon x \in X \} \subset A^C$. 
\par
Let $E,F$ be subsets of a group $G$.
We write $E F \coloneqq \{g h : g \in E, h \in F\}$
and  define inductively $E^n$ for all $n \in \N$ by setting $E^0 \coloneqq  \{1_G\}$ and $E^{n + 1} \coloneqq  E^n E$.
\par
Let $A$ be a set and let $E$ be a subset of a group $G$. 
Given $x \in A^E$, we define $gx \in A^{gE}$ by $(gx)(h) \coloneqq x(g^{-1}h)$ for all $h \in g E$. 

\subsection{Topologically transitive linear subshifts}
An action of a group $G$ on a Hausdorff topological space $X$ is called \emph{topologically transitive} 
provided that given any two nonempty open subsets $U, V \subset X$ there exists $g \in G$ such that $U \cap gV \neq \varnothing$.
Moreover, if $A$ is a set, a subshift $\Sigma \subset A^G$ is said to be \emph{topologically transitive} if
the restriction to $\Sigma$ of the $G$-shift is topologically transitive.
It is straightforward that if the acting group $G$ is infinite, then every topologically mixing $G$-action 
(e.g., any topologically mixing subshift $\Sigma \subset A^G$) is topologically transitive.
\par
The following result, which we shall use in the proof of Theorem \ref{t:char-nilpotent-finite-lca}, has some interest on its own.
We thank the referee for pointing out a gap in our original argument and providing us with an outline of the present proof.

\begin{proposition}
\label{p:TopTransLinShifts}
Let $G$ be a group and let $A$ be a vector space over a field $K$. 
Let $\Sigma \subset A^G$ be a linear subshift and suppose that $\Sigma$ is topologically transitive and finite dimensional.
Then $\Sigma = \{0\}$.
\end{proposition}
\begin{proof}
Suppose, by contradiction, that $\Sigma$ is nontrivial. Let $x_0 \in \Sigma \setminus \{0\}$ and let $g_0 \in G$ such that
$a \coloneqq x_0(g_0) \neq 0_A$. 
Consider the open subsets $U_1 \coloneqq \{x \in \Sigma: x(g_0) = 0_A\}$ and $V_1 \coloneqq \{x \in \Sigma: x(g_0) = a\}$.
Note that $U_1 \neq \varnothing$ since $0 \in U_1$ and $V_1 \neq \varnothing$ since $x_0 \in V_1$. By topological transitivity,
we can find $h_1 \in G$ such that $U_1 \cap h_1V_1 \neq \varnothing$. 
Thus, if $z_1 \in U_1 \cap h_1V_1$ and $g_1 \coloneqq h_1g_0$, we have $z_1(g_0) = 0_A$ and $z_1(g_1) = a$.
Consider now the open subsets $U_2 \coloneqq \{x \in \Sigma: x(g_0) = x(g_1) = 0_A\}$ and 
$V_1 \coloneqq \{x \in \Sigma: x(g_0) = 0 \mbox{ and } x(g_1) = a\}$.
Note that $U_2 \neq \varnothing$ since $0 \in U_2$ and $V_2 \neq \varnothing$ since $x_1 \in V_2$. By topological transitivity,
we can find $h_2 \in G$ such that $U_2 \cap h_2V_2 \neq \varnothing$. Thus, if $z_2 \in U_2 \cap h_2V_2$ and 
$g_2 \coloneqq h_2g_1$, we have $z_2(g_0) = z_2(g_1) = 0_A$ and $z_1(g_2) = a$.
Continuing this way, we inductively find $z_1, z_2, \ldots \in \Sigma$ and group elements 
$g_1, g_2, \ldots \in G$  such that 
\[
z_n(g_0) = z_n(g_1) = \cdots = z_n(g_{n-1}) = 0_A \ \mbox{ and } \ z_n(g_n) = a
\]
for all $n \geq 1$. 
As $a \not= 0_A$, it is straightforward that the configurations $z_1, z_2, \ldots$ are linearly independent, contradicting the assumption
that $\Sigma$ is finite dimensional. We deduce that $\Sigma = \{0\}$.
\end{proof}

\subsection{Linear subshifts of finite type}
We begin with two simple useful facts.

\begin{lemma}
\label{l:base-change-sft}
Let $A$ be a set and let $G$ be a group.
Let $D \subset G$ and let $P \subset A^D$.
Let $\Sigma \coloneqq  \Sigma(D, P) \subset A^G$.
Let  $E \subset G$ such that $D \subset E$.
Then one has $\Sigma = \Sigma(E, \Sigma_E)$. In particular, $\Sigma = \Sigma(D, \Sigma_D)$.
\end{lemma}
\begin{proof}
Let $x \in \Sigma$ and let $g \in G$.
Then $(g^{-1}x) \vert_E \in \Sigma_E$.
Thus $\Sigma \subset \Sigma(E, \Sigma_E)$.
Conversely, let
$x \in \Sigma(E, \Sigma_E)$. Then, for every $g \in G$, we have
$(g^{-1}x)\vert_D =  ((g^{-1}x)\vert_E)\vert_D\in (\Sigma_E)_D = \Sigma_D \subset P$, since $D \subset E$.
Therefore, $x \in \Sigma(D,P)=\Sigma$, and the conclusion follows.
\end{proof}

The following lemma states that every linear SFT admits a defining set of admissible patterns which is a vector space.

\begin{lemma}
\label{l:W-subspace}
Let $G$ be a group and let $A$ be a vector space over a field $K$.
Let $\Sigma \subset A^G$ be a linear SFT and let $D \subset G$ be a memory set for $\Sigma$.
Then there exists a vector subspace $W \subset A^D$ such that $\Sigma = \Sigma(D,W)$. 
\end{lemma}
\begin{proof}
The set $W \coloneqq \Sigma_D = \{x\vert_D: x \in \Sigma\} \subset A^D$
is a vector subspace of $A^D$ and we have $\Sigma = \Sigma(D,W)$ by Lemma \ref{l:base-change-sft}.
\end{proof}

\begin{proposition}
\label{p:FD-SFT}
Let $G$ be a finitely generated group and let $A$ be a vector space over a field $K$.
Then every finite-dimensional linear subshift $\Sigma \subset A^G$ is of finite type.
\end{proposition}
\begin{proof}
Let $S \subset G$ be a finite generating subset of $G$.
After replacing $S$ by $S \cup S^{-1} \cup \{1_G\}$, we can assume that $S = S^{-1}$ and $1_G \in S$.
Then, given any element $g \in G$, there exist $n \in \N$ and $s_1, s_2, \ldots, s_n \in S$ such that
$g = s_1s_2 \cdots s_n$. The minimal $n \in \N$ in such an expression of $g$ is the $S$-length of $g$,
denoted by $\ell_S(g)$. For every $n \in \N$ we set $B_n \coloneqq \{g \in G: \ell_S(g) \leq n\}$. 
\par
Let $\Sigma \subset A^G$ be a finite-dimensional linear subshift. 
For every $n \in \N$ denote by $\pi_n \colon \Sigma \to \Sigma_{B_n}$ the restriction map.
Note that $\pi_n$ is linear and that setting $\Sigma_n \coloneqq \ker \pi_n$ we have that
$(\Sigma_n)_{n \in \N}$ is a decreasing sequence of vector subspaces of $\Sigma$.
Now, on the one hand, since $\bigcup_{n \in \N} B_n = G$, we have $\bigcap_{n \in \N} \Sigma_n = \{0\}$. 
On the other hand, since $\Sigma$ is finite-dimensional, the above sequence
eventually stabilizes, i.e., there exists $n_0 \in \N$ such that $\Sigma_n = \Sigma_{n_0}$ for all $n \geq n_0$. 
We deduce that $\Sigma_{n_0} = \{0\}$. Thus, setting $\Delta \coloneqq B_{n_0}$, the restriction map 
$\Sigma \to \Sigma_\Delta$ is injective (in fact bijective).
\par
Set $D \coloneqq S \Delta \subset G$ and $W \coloneqq \Sigma_D \subset A^D$, and let us show that $\Sigma = \Sigma(D,W)$.
\par
Let $x \in \Sigma$. Then for every $g \in G$ we have $g^{-1}x \in \Sigma$ so that $(g^{-1}x)\vert_D \in \Sigma_D = W$.
This shows that $x \in \Sigma(D,W)$, and the inclusion $\Sigma \subset \Sigma(D,W)$ follows.
\par
Conversely, suppose that $x \in \Sigma(D,W)$. 
By definition of $\Sigma(D,W)$, for every $g \in G$, there exists $x_g \in \Sigma$ such that $(g^{-1} x)\vert_D = (x_g)\vert_D$.
Observe that, given $g \in G$, such an $x_g$ is unique since $\Delta \subset D$.
Let us show, by induction on the $S$-length of $g$, that
\begin{equation}
\label{e:x-g-g-inv-x-1-G}
x_g = g^{-1}x_{1_G}
\end{equation}
for all $g \in G$. 
If $\ell_S(g) = 0$, then $g = 1_G$ and there is nothing to prove.  
Suppose now that~\eqref{e:x-g-g-inv-x-1-G} is satisfied for all $g \in G$ such that $\ell_S(g) = n$
and let $h \in G$ such that $\ell_S(h) = n + 1$.
Then there exist $g \in G$ with $\ell_S(g) = n$ and $s \in S$ such that $h = g s$.
For all $d \in \Delta$, we have
\begin{align*}
x_h(d) &= (h^{-1} x)(d) && \text{(since $\Delta \subset D$)} \\ 
&= (g^{-1} x)(s d) && \text{(since $h = g s$)} \\ 
&= x_g(s d) && \text{(since $S \Delta \subset D$)} \\
&= (g^{-1} x_{1_G})(s d) && \text{(by our induction hypothesis)} \\ 
&= (h^{-1} x_{1_G})(d) && \text{(since $h = g s$).}
\end{align*}
Thus $x_h$ and $h^{-1} x_{1_G}$ coincide on $\Delta$.
As $x_h, h^{-1} x_{1_G} \in \Sigma$, this implies that $x_h = h^{-1} x_{1_G}$.
By induction, we conclude that~\eqref{e:x-g-g-inv-x-1-G} holds for all $g \in G$. 
Since $1_G \in D$, we deduce that
\[ 
x(g) = (g^{-1}x)(1_G) = x_g(1_G) = (g^{-1}x_{1_G})(1_G) = x_{1_G}(g)
\]
for all $g \in G$.
This shows that $x = x_{1_G} \in \Sigma$, and the inclusion $\Sigma(D,W) \subset \Sigma$ follows.
\par
In conclusion, $\Sigma = \Sigma(D,W)$ is a subshift of finite type.
\end{proof}

The condition that $G$ is finitely generated cannot be removed from the assumptions in Proposition \ref{p:FD-SFT}.
In fact we have the following (cf.\ \cite[Lemma 1]{salo}; see also Section~\ref{ss:GLMT}).

\begin{corollary}
\label{c:constant-conf-SFT}
Let $G$ be a group and let $A$ be a nontrivial finite dimensional vector space over a field $K$. 
Consider the subshift $\Sigma \subset A^G$ consisting of all constant configurations.
Then $\Sigma$ is a SFT if and only if $G$ is finitely generated.
\end{corollary}

\begin{proof}
As the linear subshift $\Sigma$ satisfies $\dim_K(\Sigma) = \dim_K(A) < \infty$,
it is a SFT whenever $G$ is finitely generated by Proposition~\ref{p:FD-SFT}.
\par
Conversely, suppose that $\Sigma$ is an SFT.
Thus there exists a finite subset $D \subset G$ and $P \subset A^D$ such that $\Sigma = \Sigma(D,P)$.
Consider the subgroup $H \subset G$ generated by $D$ and let $a \in A$ such that $a \not= 0$. 
Then the configuration $x \in A^G$ such that $x(g) = 0$ if $g \in H$ and $x(g) = a$ otherwise 
belongs to $\Sigma(D,P)$ since, for each $g \in G$, either $gD \subset H$ or $gD \subset G \setminus H$.
As $\Sigma(D,P) = \Sigma$ and every configuration in $\Sigma$ is constant, we conclude that $H = G$.
Therefore $G$ is finitely generated.
\end{proof}

\subsection{Restriction of linear-sofic subshifts and of linear CAs}
\label{s:restriction-ls} 
Let $G$ be a group and let $A$ be a vector space over a field $K$. 
Recall that a linear subshift $\Sigma \subset A^G$ is said to be a \emph{linear-sofic subshift} 
if there exists a vector space $B$ over $K$,
an SFT $\Sigma' \subset B^G$, and a linear cellular automaton $\tau \colon B^G \to A^G$ such that
$\Sigma = \tau(\Sigma')$. We shall refer to a finite subset $M \subset G$ containing both a memory set for $\Sigma'$ 
as well as a memory set for $\tau$ as to a \emph{memory set} for the  linear-sofic subshift $\Sigma$.
\par
Let $\Sigma \subset A^G$ be a linear-sofic subshift. 
Let $H\subset G$ be a subgroup of $G$ containing a memory set for $\Sigma$.  
Denote by $G/H \coloneqq \{gH: g \in G\}$ the set of all right cosets of $H$ in $G$.
As the right cosets of $H$ in $G$ form a partition of $G$,
we have a natural factorization
\[
A^G = \prod_{c \in G/H} A^c
\]
in which each $x \in A^G$ is identified with $(x\vert_c)_{c \in G/H} \in \prod_{c \in G/H} A^c$. 
The above factorization of $A^G$ induces a factorization (cf.~\cite[Lemma 2.8]{ccp-2020})
\[
\Sigma = \prod_{c \in G/H} \Sigma_c,
\]
where $\Sigma_c = \{x\vert_c: x \in \Sigma\}$ is a vector subspace of $A^c$ for all $c \in G/H$. 
\par 
Let $T \subset G$ be a complete set of representatives for the right cosets of $H$ in $G$ such that $1_G \in T$.
Then, for each $c \in G/H$, we have a linear uniform homeomorphism 
$\phi_c \colon \Sigma_c \to \Sigma_H$ given by
$\phi_c(y)(h) = y(gh)$ for all $y \in \Sigma_c$, where $g \in T$ represents $c$.
\par 
Now suppose in addition that $\tau \colon \Sigma \to \Sigma$ is a linear CA
which admits a memory set contained in $H$.
Then we have $\tau = \prod_{c \in G/H} \tau_c$,
where $\tau_c \colon \Sigma_c \to \Sigma_c$ is the linear map defined by setting
$\tau_c(y) \coloneqq \tau(x)\vert_c$ for all $y \in \Sigma_c$,
where $x \in \Sigma$ is any configuration extending $y$. 
Note that for each $c \in G/H$, the linear maps $\tau_c$ and $\tau_H$ are conjugated by $\phi_c$, i.e., we have 
$\tau_c = \phi_c^{-1} \circ \tau_H \circ \phi_c$.
This allows us to identify the action of $\tau_c$ on $\Sigma_c$
with that of the \emph{restriction cellular automaton} $\tau_H$ on $\Sigma_H$.

The following extends \cite[Theorem 2.1]{cc-TCS} (cf.~\cite[Lemma 2.10]{ccp-2020}).

\begin{lemma}
\label{l:restriction-cip}
Let $G$ be a group, let $A$ be a vector space over a field $K$, and let $\Sigma \subset A^G$ be a linear-sofic subshift.
Let $\tau \colon A^G \to A^G$ be a cellular automaton.
Let $H \subset G$ be a subgroup containing memory sets for both $\Sigma$ and $\tau$. 
Then $\tau(\Sigma)$ is closed in $A^G$ if and only if $\tau_H(\Sigma_H)$ is closed in $A^H$. 
\end{lemma}
 
\begin{proof}
With the above notation, we have $\tau(\Sigma) = \prod_{c \in G/H} \tau_c(\Sigma_c)$. 
It is immediate that $\tau_H(\Sigma_H)$ is closed in $A^H$ if $\tau(\Sigma)$ is closed in $A^G$. 
Conversely, if $\tau_H(\Sigma_H)$ is closed in $A^H$, then 
so are $\tau_c(\Sigma_c) = \phi_c^{-1}(\tau_H(\Sigma_H))$ in $A^c$ for all $c \in G/H$, 
since the $\phi_c \colon A^c \to A^H$ are uniform homeomorphisms.  
Consequently, $\tau(\Sigma)$ is closed in $A^G$ whenever $\tau_H(\Sigma_H)$ is closed in $A^H$ 
since the product of closed subspaces is closed in the product topology. 
\end{proof}

\subsection{Nilpotent linear cellular automata}
Let $G$ be a group, let $A$ be a vector space over a field $K$, and let $\tau \colon \Sigma \to \Sigma$ be a linear cellular automaton, 
where $\Sigma \subset A^G$ is a linear subshift.
By linearity, $\tau$ is nilpotent if and only if there exists and integer $n_0 \geq 1$ such that $\tau^{n_0} = 0$.
Moreover, $\tau^n(\Sigma)$, $n \in \N$, and therefore $\Omega(\tau)$ are vector subspaces of $\Sigma$. 
\par
The following is the linear version of~\cite[Lemma 2.9]{ccp-2020}.

\begin{lemma}
\label{l:restriction-ls}
Let $G$ be a group, let $A$ be a vector space over a field $K$, and let $\Sigma \subset A^G$ be a linear-sofic subshift.
Let $\tau \colon \Sigma \to \Sigma$ be a linear cellular automaton.
Let $H \subset G$ be a subgroup containing memory sets of both $\Sigma$ and $\tau$. 
Then the following hold:
\begin{enumerate} [{\rm (i)}]
\item $\Omega(\tau) = \prod_{c \in G/H} \Omega(\tau_c)$;
\item $\Omega(\tau)$ is linearly uniformly homeomorphic to $\Omega(\tau_H)^{G/H}$;
\item
$\tau$ is nilpotent if and only if $\tau_H \colon \Sigma_H \to \Sigma_H$ is nilpotent.
\end{enumerate}
\end{lemma}

\begin{proof}
We have $\tau^n(\Sigma) = \prod_{c \in G/H} \tau_c^n(\Sigma)$ for all $n \in \N$, so that
\[
\Omega(\tau) = \bigcap_{n \in \N} \tau^n(\Sigma) = \bigcap_{n \in \N}\prod_{c \in G/H} \tau_c^n(\Sigma_c) =
\prod_{c \in G/H}  \bigcap_{n \in \N} \tau_c^n(\Sigma_c) = \prod_{c \in G/H} \Omega(\tau_c).
\]
This proves (i). It is then clear that $\phi \coloneqq \prod_{c \in G/H} \phi_c \colon A^G \to (A^H)^{G/H}$ yields, by restriction,
a linear uniform homeomorphism $\Omega(\tau) = \prod_{c \in G/H} \Omega(\tau_c) \to \Omega(\tau_H)^{G/H}$.
This proves (ii).  
\par
We have that $\tau$ is nilpotent if and only if there exists an integer $n_0 \geq 1$ such that $\tau^{n_0}(\Sigma) = \{0\}$.
By the above discussion, this is equivalent to $\tau_H^{n_0}(\Sigma_H) = \{0\}$, that is, to $\tau_H$ being nilpotent.
\end{proof}

\par
Given a group $G$ and a vector space $A$ over a field $K$, the set $\LCA(G,A)$ of all linear cellular automata $\tau \colon A^G \to A^G$
has a natural structure of a $K$-algebra (cf.\ \cite[Section 8.1]{book}). Indeed, it is a subalgebra of the $K$-algebra 
$\End_{K[G]}A^G$ (\cite[Proposition 8.7.1]{book}).
In \cite[Section 6]{cc-ETDS1} and \cite[Section 4]{cc-IJM} (see also \cite[Section 4]{cc-JA} and \cite[Corollary 8.7.8]{book}) 
it is shown that if $A$ is finite-dimensional with, say, $\dim_K(A) = d$, 
then, once fixed a vector basis $\BB$ for $A$, there exists a canonical $K$-algebra isomorphism $\tau \mapsto M_\BB(\tau)$ of 
$\LCA(G,A)$ onto $\Mat_d(K[G])$, the $K$-algebra of $d \times d$ matrices with coefficients in the group ring $K[G]$.
\par
Recall that an element $M$ of a ring $R$ is called \emph{nilpotent} if there exists an integer $n \geq 1$ such that $M^n = 0_R$.
We then have 

\begin{proposition}
\label{p:char-nilpotent-finite-lca}
Let $G$ be a group and let $A$ be a finite-dimensional vector space over a field $K$. 
Suppose that $\dim_K(A) = d$ and that $A$ is equipped with a basis $\BB$.
Let $\tau \colon A^G \to A^G$ be a linear CA.  
Then the following conditions are equivalent: 
\begin{enumerate}[\rm (a)]
\item
$\tau$ is nilpotent;
\item
the matrix $M_\BB(\tau) \in \Mat_d(K[G])$ is nilpotent.
\end{enumerate}
\end{proposition}
\begin{proof}
The proof follows immediately from the fact that $\LCA(G,A)$ and $\Mat_d(K[G])$ are isomorphic as $K$-algebras, and that nilpotency is preserved under $K$-algebra homomorphisms.
\end{proof}

\section{Space-time inverse systems of linear cellular automata} 
\label{s:STIS}
\subsection{Inverse limits of sets}
\label{s:countably-pro-constructible}
Let $(I, \preceq)$ be a directed set, i.e., a partially ordered set in which every pair of elements admits an upper bound.
An \emph{inverse system} of sets \emph{indexed} by $I$ consists of the following data:
(1) a set $Z_i$ for each $i \in I$;
(2) a \emph{transition map} $\varphi_{ij} \colon Z_j \to Z_i$
for all $i,j \in I$ such that $i \preceq j$.
Furthermore, the transition maps must satisfy the following conditions:
\begin{align*}
 \varphi_{ii} &= \Id_{Z_i} \text{ (the identity map on $Z_i$) for all } i \in  I, \\[4pt]
 \varphi_{ij} \circ \varphi_{jk}  &= \varphi_{ik}  \text{ for all $i,j,k \in I$ such that } i \preceq j \preceq k.
\end{align*}
One then speaks of the inverse system $(Z_i,\varphi_{ij})$, or simply  $(Z_i)$ if
the index set and the transition maps  are clear from the context. One says that
an inverse system $(Z_i,\varphi_{ij})$ satisfies the \emph{Mittag-Leffler condition} provided
that for each $i \in I$ there exists $j \in I$ with $i \preceq j$ such that 
$\varphi_{i k} (X_k) = \varphi_{i j} (X_j)$ for all $j \preceq k$. 
\par
The \emph{inverse limit} of an inverse system $(Z_i,\varphi_{i j})$ is the subset
\[
\varprojlim_{i \in I} (Z_i,\varphi_{i j}) =  \varprojlim_{i \in I} Z_i \subset \prod_{i \in I} Z_i
\]
consisting of all  $(z_i)_{i \in I}$ such that $\varphi_{i j}(z_j)= z_i$ for all $i \preceq j$. 
\par
 
The following useful lemma is an application of the classical Mittag-Leffler lemma to affine inverse systems
(see, e.g.\ \cite[Section I.3]{grothendieck}, where Grothendieck used it in his study of the cohomology of affine schemes,
the general treatment by Bourbaki \cite[Theorem 1, TG II. Section 5]{bourbaki},
or \cite[Lemma 3.1]{cc-TCS} for a self-contained proof in the countable case. Cf.\ also \cite[Proposition 4.2]{phung-2018} and 
\cite[Lemma~9.15]{phung-2020} for more details).

\begin{lemma}
\label{l:inverse-limit-closed-im}
Let $K$ be a field.
Let $(X_i, f_{ij})$ be an inverse system indexed by an index set $I$, 
where  each $X_i$ is a nonempty finite-dimensional $K$-affine space and each transition map  
$f_{i j} \colon X_j \to X_i$ is a $K$-affine map for all $i \preceq j$. 
Then $\varprojlim_{i \in I} X_i \neq \varnothing$.
\end{lemma}

\subsection{Space-time inverse systems}  
\label{s:space-time-system} 

Let $G$ be a group and let $A$ be a vector space over a field $K$. 
Let $\Sigma \subset A^G$ be a linear subshift and let $\tau \colon \Sigma \to \Sigma$ be a linear CA.
Let $\widetilde{\tau} \colon A^G \to A^G$ be a linear CA extending $\tau$ and let $M\subset G$ be a memory set for $\widetilde{\tau}$.
Since every finite subset of $G$ containing a memory set for $\widetilde{\tau}$ is itself a memory set for $\widetilde{\tau}$,
we can choose $M$ such that $1_G \in M$ and $M = M^{-1}$.
\par
Let ${\mathcal P}^*(G)$ denote the set of all finite subsets of $G$ containing $1_G$ equipped with the ordering
given by inclusion. Also equip $\N$ with the natural ordering.
Equip $I \coloneqq {\mathcal P}^*(G) \times \N$ with the product ordering $\preceq$. Thus, given $\Omega, \Omega' \in {\mathcal P}^*(G)$
and $n,n' \in \N$, we have $(\Omega,n) \preceq (\Omega',n')$ if and only if $\Omega \subset \Omega'$ and $n \leq n'$.
It is clear that $(I, \preceq)$ is directed.
\par
We construct an inverse system $(\Sigma_{\Omega,n})_{(\Omega,n) \in I}$ indexed by $I$ as follows.
\par
Firstly, given $(\Omega,n) \in I$, we set 
\[
\Sigma_{\Omega,n} \coloneqq \Sigma_{\Omega M^n} = \{x\vert_{\Omega M^n} : x \in \Sigma \} \subset A^{\Omega M^n}.
\]

To define the transition maps
$\Sigma_{\Omega',n'} \to \Sigma_{\Omega,n}$ ($(\Omega,n) \preceq (\Omega',n')$) 
of the inverse system $(\Sigma_{\Omega,n})_{(\Omega,n) \in I}$,
it is clearly enough to define, for all $\Omega, \Omega'\in {\mathcal P}^*(G)$ and $n,n' \in \N$,  
with $\Omega\subset \Omega'$ and $n \leq n'$, the \emph{horizontal} transition map 
$p_{\Omega, \Omega'; n} \colon \Sigma_{\Omega', n}  \to \Sigma_{\Omega, n}$,
the \emph{vertical} transition map $q_{\Omega; n, n'} \colon \Sigma_{\Omega, n'} \to \Sigma_{\Omega,n}$,
and verify that the diagram 
\[
\begin{tikzcd}
\Sigma_{\Omega,n'} \ \arrow[d, swap, "q_{\Omega; n,n'}"]  
&  \ \Sigma_{\Omega',n'}  \arrow[l, swap,  "p_{\Omega, \Omega'; n'}"]  \arrow[d, "q_{\Omega'; n,n'}"]   \\ 
\Sigma_{\Omega,n}  \  &  \ \Sigma_{\Omega', n} \arrow[l, "p_{\Omega, \Omega'; n}"],   
\end{tikzcd}
\] 
is commutative, i.e.,
\begin{equation}
\label{e:p-q-commute}
q_{\Omega;n,n'} \circ p_{\Omega, \Omega';n'} = p_{\Omega, \Omega'; n} \circ q_{\Omega';n,n'}
\end{equation}
for all $\Omega, \Omega' \in {\mathcal P}^*(G)$ and $n,n' \in \N$, with $\Omega\subset \Omega'$ and $n \leq n'$.

We define $p_{\Omega, \Omega';n}$ as being the linear map obtained by restriction to $\Omega M^n \subset \Omega' M^n$.
Thus, for all $\sigma \in \Sigma_{\Omega',n} = \Sigma_{\Omega' M^n}$, we have 
\begin{equation}
\label{e:def-p-ij}
p_{\Omega, \Omega';n}(\sigma) = \sigma\vert_{\Omega M^n}. 
\end{equation} 
We now define $q_{\Omega; n,n'}$. If $n = n'$, then 
$q_{\Omega; n,n'} = q_{\Omega; n,n} \coloneqq \Id_{\Omega M^n} \colon \Omega M^n \to \Omega M^n$, is the identity map.
Suppose now that $n+1 \leq n'$. We first observe that, given $x \in \Sigma$ and $g \in G$,
it follows from~\eqref{e;local-property} applied to $\widetilde{\tau}$ that
$\tau^{n'-n}(x)(g)$ only depends on the restriction of $x$ to $gM^{n'-n}$.
As $gM^{n'-n} \subset \Omega M^n M^{n'-n} = \Omega M^{n'}$ for all $g \in \Omega M^n$,
we deduce from this observation that, given $\sigma \in \Sigma_{\Omega, n'} = \Sigma_{\Omega M^{n'}}$ and $x \in \Sigma$ 
extending $\sigma$, the formula
\begin{equation}
\label{e:def-q-ij}
q_{\Omega; n, n'}(\sigma) \coloneqq \tau^{n'-n}(x)\vert_{\Omega M^n} 
\end{equation}
yields a well-defined element $q_{\Omega; n, n'}(\sigma) \in \Sigma_{\Omega M^n} = \Sigma_{\Omega, n}$, and hence a linear map
$q_{\Omega; n,n'} \colon \Sigma_{\Omega, n'} \to \Sigma_{\Omega, n}$.

\begin{definition}
The inverse system $(\Sigma_{\Omega,n})_{(\Omega,n) \in I}$ is called the \emph{space-time inverse system} 
associated with the triple $(\Sigma, \tau, M)$. 
\end{definition} 

When $G$ is countable, we can simplify the above construction by slightly modifying the definitions therein.
Since $G$ is countable, we can find a sequence $(M_n)_{n \in \N}$ of finite subsets of $G$ such that
\begin{enumerate}[{\rm (M-1)}]
\item $M_0 = \{1_G\}$ and $M_1 = M$ (the memory set for $\widetilde{\tau}$), 
\item $M_iM_j \subset M_{i+j}$ for all $i,j \in \N$, 
\item $\bigcup_{n \in \N} M_n = G$.
\end{enumerate}
For instance, if $G$ is finitely generated and $M$ in addition generates $G$, then one may take $M_n \coloneqq M^n$ for all $n \in \N$.
We equip $\N^2$ with the product ordering $\preceq$, that is, given $i,j,k,l \in \N$, we have  $(i,j) \preceq (k,l)$ if and only if $i \leq k$ and $j \leq l$. We then construct an inverse system $(\Sigma_{i j})_{i,j \in \N}$ indexed by the directed set $(\N^2, \preceq)$ by setting
\[
\Sigma_{i j} \coloneqq \Sigma_{M_{i+j}} = \{x\vert_{M_{i  + j}} : x \in \Sigma \} \subset A^{M_{i + j}}
\]
and defining, for all $i,j \in \N$, the \emph{unit-horizontal} transition map $p_{i j} \colon \Sigma_{i+1, j}  \to \Sigma_{i j}$
as being the linear map obtained by restriction to $M_{i + j} \subset M_{i + j + 1}$,
and the \emph{unit-vertical} transition map $q_{i j} \coloneqq q_{M_i;M_j;M_{j+1}} \colon \Sigma_{i,j+1} \to \Sigma_{i j}$ by setting
$q_{i j}(\sigma) \coloneqq (\tau(x))\vert_{M_{i + j}}$ for all $\sigma \in \Sigma_{i,j + 1}$ and $x \in \Sigma$ extending $\sigma$
(as in the general case, this gives a well-defined element $q_{i j}(\sigma) \in \Sigma_{i j}$).
Finally, as for \eqref{e:p-q-commute}, one checks that $q_{i j} \circ p_{i,j+1} = p_{i j} \circ q_{i+1,j}$ for all $i,j \in \N$. 

\begin{definition}
\label{d:space-time-system-2}  
The inverse system $(\Sigma_{ij})_{i,j \in \N}$ is called the \emph{space-time inverse system} 
associated with the triple $(\Sigma, \tau, (M_n)_{n \in \N})$. 
\end{definition}

\subsection{Space-time-systems and limit sets}
We keep the assumptions and notation from the above subsection.
Let us fix $n \in \N$. Then, in our space-time inverse system we get an \emph{horizontal} inverse system 
$(\Sigma_{\Omega,n})_{\Omega \in {\mathcal P}^*(G)}$ indexed by ${\mathcal P}^*(G)$
whose transition maps are the restriction maps 
$p_{\Omega, \Omega';n} \colon \Sigma_{\Omega' M^n} \to \Sigma_{\Omega M^n}$, 
$\Omega, \Omega' \in {\mathcal P}^*(G)$ such that $\Omega \subset \Omega'$.
Note that the horizontal inverse system satisfies the Mittag-Leffler condition and that in fact, as it
immediately follows from the closedness of $\Sigma$ in $A^G$ and the fact that $G M^n = G$, one has that the~limit
\begin{equation}
\label{e:row-limit}
\Sigma_n \coloneqq \varprojlim_{\Omega \in {\mathcal P}^*(G)} \Sigma_{\Omega,n} 
\end{equation}
can be identified with $\Sigma$ in a canonical way.
\par
Moreover, the linear maps $q_{\Omega; n,n'} \colon \Sigma_{\Omega,n'} \to \Sigma_{\Omega,n}$, for $\Omega \in {\mathcal P}^*(G)$,
define an inverse system linear morphism from
the inverse system $(\Sigma_{\Omega,n'})_{\Omega \in {\mathcal P}^*(G)}$ to the inverse system 
$(\Sigma_{\Omega,n})_{\Omega \in {\mathcal P}^*(G)}$.
This yields a linear limit map $\tau_{n,n'} \colon \Sigma_{n'} \to \Sigma_n$.
Using the identifications $\Sigma_n  = \Sigma_{n'} = \Sigma$, we have $\tau_{n,n'} = \tau^{n'-n}$ 
for all $n,n' \in \N$ such that $n \leq n'$.
We deduce that the limit
\begin{equation}
\label{e:backward-orbits}
\varprojlim_{(\Omega, n) \in I} \Sigma_{\Omega, n} = \varprojlim_{n \in \N} \Sigma_n
\end{equation}
is the set of \emph{backward orbits} of $\tau$, that is, 
the set consisting of all sequences $(x_n)_{n \in \N}$ such that $x_n \in \Sigma$ and 
$x_n = \tau(x_{n+1})$ for all $n \in \N$.
Each such a sequence satisfies that $x_0 = \tau^n(x_n)$ for all $n \in \N$, and hence $x_0 \in \Omega(\tau)$.
This determines a canonical linear map 
\[
\Phi \colon \varprojlim_{(\Omega, n) \in I} \Sigma_{\Omega, n} \to \Omega(\tau).
\] 
 
\begin{proposition} 
\label{p:phi-surg}
Let $G$ be a group and let $A$ be a finite dimensional vector space over a field $K$. 
Let $\Sigma \subset A^G$ be a linear subshift and let $\tau \colon \Sigma \to \Sigma$ be a linear CA.
Let $M \subset G$ be a memory set for $\tau$ which is symmetric and contains $1_G$, 
and consider the space-time inverse system associated with the triple $(\Sigma,\tau, M)$. 
Then the canonical map $\Phi$ is surjective.
\end{proposition} 

\begin{proof}
Let $y_0 \in \Omega(\tau) \subset \Sigma$. 
For every $\Omega \in {\mathcal P}^*(G)$ and $n \in \N$, define a finite dimensional affine subspace $B_{\Omega, n} \subset A^{\Omega M^n}$ by setting
\[
B_{\Omega, n}  \coloneqq \left(q_{\Omega;0,1} \circ q_{\Omega;1,2} \circ \dots \circ q_{\Omega; n-1,n}\right)^{-1} (y_0\vert_\Omega) \subset \Sigma_{\Omega M^n}. 
\] 
By definition of $\Omega(\tau)$, for every $n \in \N$ there exists 
an element $y_n \in \Sigma$ such that $\tau^n (y_n)=y_0$. 
Hence, it follows from the definition of the transition maps 
$q_{\Omega; j-1,j}$ and of $\Sigma_{\Omega, j}$, $j \in \N$, that $y_n\vert_{\Omega M^n} \in B_{\Omega,n}$. 
In particular, $B_{\Omega,n} \neq \varnothing$ for every $\Omega \in {\mathcal P}^*(G)$ and $n \in \N$. 
By restricting the transition maps of the space-time inverse system $(\Sigma_{\Omega, n})_{(\Omega, n) \in I}$ 
to the sets $B_{\Omega, n}$, we obtain a well-defined inverse subsystem 
$(B_{\Omega, n})_{(\Omega, n) \in I}$ of finite dimensional affine spaces with affine transition maps.
By Lemma~\ref{l:inverse-limit-closed-im}, we can find   
\[
x  \in \varprojlim_{(\Omega, n) \in I} B_{\Omega, n}  \subset \varprojlim_{(\Omega, n) \in I} \Sigma_{\Omega, n}. 
\] 
It is clear from the constructions of the inverse system $(B_{\Omega, n})_{(\Omega, n) \in I}$ 
and of the map $\Phi$ that $\Phi(x)= y_0$. 
This shows that $\Phi$ is surjective. 
\end{proof}

\section[The closed-image-property]{The closed-image-property for linear cellular automata}

Using the space-time inverse system, we give a short proof of the 
following result extending \cite[Theorem 1.4]{cc-TCS} (see also \cite[Theorem 8.8.1]{book}).

\begin{theorem}
\label{t:closed-image}
Let $G$ be a group and let $A$ be a finite-dimensional vector space over a field $K$.  
Let $\Sigma \subset A^G$ be a linear subshift and let $\tau \colon A^G \to A^G$ be a linear CA. 
Then $\tau(\Sigma)$ is a linear subshift of $A^G$.  
\end{theorem}

\begin{proof} 
Since the cellular automaton $\tau$ is linear and $G$-equivariant, its image $\tau(\Sigma)$ is a $G$-invariant vector subspace of $A^G$.
We thus only need to show that $\tau(\Sigma)$ is closed in $A^G$.
Let $M \subset G$ be a memory set for $\tau$ which is symmetric and contains $1_G$, 
and consider the space-time inverse system associated with the triple $(\Sigma, \tau, M)$ 
as in Section~\ref{s:space-time-system}.
\par
Suppose that $x \in \Sigma$ belongs to the closure of $\tau(\Sigma)$. We must show that $x \in \tau(\Sigma)$. 
\par 
For every $\Omega \in {\mathcal P}^*(G)$, define an affine subspace $Z_\Omega \subset A^{\Omega M}$ by setting
\[ 
Z_\Omega \coloneqq (q_{\Omega;0,1})^{-1}(x\vert_\Omega) \cap \Sigma_{\Omega,1}.
\] 

Since $x$ belongs to the closure of $\tau(\Sigma)$, it follows that $Z_\Omega \neq \varnothing$ for all $\Omega \in {\mathcal P}^*(G)$. 
By restricting the projections $p_{\Omega, \Omega';1} \colon A^{\Omega'M} \to A^{\Omega M}$ (cf.\ \eqref{e:def-p-ij},
$\Omega, \Omega' \in  {\mathcal P}^*(G)$, with $\Omega \subset \Omega'$) to the $Z_\Omega$'s, we obtain affine maps $\pi_{\Omega, \Omega'} \colon Z_{\Omega'} \to Z_{\Omega}$ 
of the inverse system $(Z_\Omega)_{\Omega \in {\mathcal P}^*(G)}$. 
It then follows from Lemma \ref{l:inverse-limit-closed-im} that $\varprojlim_{\Omega \in {\mathcal P}^*(G)} (Z_\Omega,\pi_{\Omega, \Omega'}) \neq \varnothing$. 
Therefore, by construction of $Z_\Omega$ and $\Sigma_{\Omega,1}$, for every $c \in \varprojlim_{\Omega \in {\mathcal P}^*(G)}Z_\Omega \subset \varprojlim_{\Omega \in {\mathcal P}^*(G)} \Sigma_{\Omega,1} = \Sigma$ (cf.~\eqref{e:row-limit}) we have $\tau(c) = x$.     
This shows that $\tau(\Sigma)$ is closed.
\end{proof}

\begin{remark}
We observe that the the hypothesis of finite-dimensionality of the vector space $A$ in Theorem \ref{t:closed-image} cannot be dropped, as
the example in Section \ref{s:CIP-examples} below shows.
\end{remark}

\section{Proofs}
\label{s:proofs}

\subsection*{Proof of Theorem \ref{t:LSFT-DCC}}
Suppose $\Sigma$ is of finite type. Hence $\Sigma = \Sigma(D,W) \subset A^D$, where $D \subset G$ is finite and
$W \subset A^D$ is a vector subspace (cf.\ Lemma \ref{l:W-subspace}).
Let $\Sigma_0 \supset \Sigma_1 \supset \cdots$ be a decreasing sequence of linear subshifts of $A^G$
such that $\bigcap_{n \geq 0} \Sigma_n = \Sigma$. Let  $(M_n)_{n \in \N}$
be a sequence of finite subsets of $G$ satisfying conditions (M-1)-(M-3) and such that $D \subset M_1$.
Consider the inverse system $(X_{ij})_{i,j \in \N}$ defined by setting
$X_{ij} \coloneqq (\Sigma_j)_{M_i} \subset A^{M_i}$. Observe that $X_{i,j+1} \subset X_{ij}$ since $\Sigma_{j+1} \subset \Sigma_j$
for all $i,j \in \N$. Also, we define the transition maps $p_{ij} \colon X_{i+1,j} \to X_{ij}$ by setting
$p_{ij}(x) \coloneqq x\vert_{M_i}$ for all $x \in X_{i+1,j} = (\Sigma_j)_{M_{i+1}}$ and
$q_{ij} \colon X_{i,j+1} \to X_{ij}$ as the inclusion maps.
\par
The decreasing sequence $(X_{1,j})_{j \in \N}$ of finite-dimensional vector spaces eventually stabilizes so that there exists $j_0 \geq 1$ such that $X_{1,j} = X_{1,j_0}$ for all $j \geq j_0$. Set $W' \coloneqq X_{1,j_0}$ and let us show that $\Sigma$ equals the linear SFT
$\Sigma' \coloneqq \Sigma(M_{1}, W')$.
First note that $\Sigma_{j_0} \subset \Sigma'$ so that $\Sigma \subset \Sigma'$.
Conversely, let $w \in W'$.  
We construct an inverse subsystem $(Z_{ij})_{i \geq 1, j \geq 0}$ of $(X_{ij})_{i \geq 1, j \geq 0}$ as follows.
For $i \geq 1$ and $j \geq 0$, consider the affine subspace of $X_{ij}$:  
\[
Z_{ij} \coloneqq   \{x \in X_{ij}: x\vert_{M_1} = w\} \subset X_{ij}.  
\]
The transition maps of $(Z_{ij})_{i \geq 1, j  \geq 0}$ are well-defined as the restrictions of the transition maps of
the system $(X_{ij})_{i \geq 1, j \geq 0}$.  
\par
By our construction, each $Z_{ij}$ is clearly nonempty.
Hence, Lemma~\ref{l:inverse-limit-closed-im} implies that
there exists $x = (x_{ij})_{i \geq 1, j \geq 0} \in \varprojlim Z_{ij}$.
Let $y \in A^G$ be defined by $y(g)= x_{i0}(g)$ for every $g \in G$ and
any large enough $i\geq 1$ such that $g \in M_{i}$.
Observe that $x_{ij} = x_{ik}$ for every $i \geq 1$ and $0 \leq j \leq k$  
since the vertical transition maps $X_{ik} \to X_{ij}$ are simply inclusions.
Consequently, for every $n \in \N$, we have $y \in \Sigma_n$ by \eqref{e:row-limit}.
Hence $y \in \Sigma$.
By construction, $y\vert_{M_{1}} = w$.
Since $w$ was arbitrary, this shows that $W' \subset \Sigma_{M_{1}}$.
Hence, $\Sigma' = \Sigma(M_{1}, W') \subset \Sigma(M_{1}, \Sigma_{M_{1}}) = \Sigma$.
The last equality follows from Lemma~\ref{l:base-change-sft} as $D \subset M_{1}$.  
Therefore, $\Sigma' = \Sigma$ and  
$\Sigma_n = \Sigma$ for all $n \geq j_0$. This proves the implication (a) $\implies$ (b). 
\par
Suppose now that $\Sigma \subset A^G$ is a linear subshift which is not of finite type. 
Let $(M_n)_{n \in \N}$ be a sequence of finite subsets of $G$ satisfying conditions (M-2)-(M-3). 
For every $n \in \N$, set $W_{n} \coloneqq \Sigma_{M_n}$ (as in Section~\ref{s:space-time-system}). 
Then, $W_n$ is a vector subspace of $A^{M_n}$. 
For every $n \in \N$ we consider the linear SFT $\Sigma_n \coloneqq \Sigma(M_n, W_n)$.
As $(\Sigma_{M_{n+1}})_{M_n} = \Sigma_{M_n}$, 
it is clear that $\Sigma \subset \Sigma_{n+1} \subset \Sigma_n$ for all $n \in \N$. 
We claim that $\Sigma = \bigcap_{n \in \N} \Sigma_n$. 
We only need to prove that $\bigcap_{n \in \N} \Sigma_n \subset \Sigma$. 
Let $x \in \bigcap_{n \in \N} \Sigma_n$. Then by definition of $\Sigma_n$, 
we find that $x\vert_{M_n} \in W_n = \Sigma_{M_n}$ for every $n \in \N$. 
Thus, since $\Sigma$ is closed, $x \in \varprojlim_{n \in \N} \Sigma_{M_n} = \Sigma$ 
(cf.~\eqref{e:row-limit}) and hence $\bigcap_{n \in \N} \Sigma_n \subset \Sigma$. 
However, the decreasing sequence $(\Sigma_n)_{n \in \N}$ cannot stabilize since, otherwise, the subshift  
$\Sigma$ would be of finite type. 
This shows that (b)$\implies$(a). 
The proof of Theorem \ref{t:LSFT-DCC} is complete.  \hfill $\Box$

\subsection*{Proof of Corollary \ref{c:DCC-FT}}
Suppose first that $A^G$ satisfies condition (b) and let $\Sigma \subset A^G$ be a linear subshift. 
Let $(D_n)_{n \in \N}$ be an increasing sequence of finite subsets of $G$ such that $\bigcup_{n \in \N} D_n = G$.
For every $n \in \N$ let $W_n \coloneqq \Sigma_{D_n} \subset A^{D_n}$. Then $\Sigma_n \coloneqq \Sigma(D_n,W_n) \subset A^G$ is a 
linear SFT and $\Sigma_0 \supset \Sigma_1 \supset \cdots \Sigma_n \supset \Sigma_{n+1} \cdots$ 
for all $n \in \N$. We claim that $\bigcap_{n \in \N} \Sigma_n = \Sigma$. 
Since $\Sigma_n \supset \Sigma$ for all $n \in \N$, we only need to show that $\bigcap_{n \in \N} \Sigma_n  \subset \Sigma$. 
Let $x \in \bigcap_{n \in \N} \Sigma_n$. 
This means that for each $n \in \N$ there exists $x_n \in \Sigma$ such that $x\vert_{D_n} = x_n\vert_{D_n}$. 
Since the sequence $(D_n)_{n \in \N}$ is exhausting and $\Sigma$ is closed in the prodiscrete topology, we deduce that $x \in \Sigma$.
This proves the claim.
Since $A^G$ satisfies condition (b), there exists $n_0 \in \N$ such that $\Sigma_n = \Sigma_{n_0}$
for all $n \geq n_0$. We deduce that $\Sigma = \Sigma_{n_0}$ is of finite type.
\par
Conversely, suppose that every linear subshift $\Sigma \subset A^G$ is of finite type and
let $(\Sigma_n)_{n \in \N}$ be a decreasing sequence of linear subshifts. 
Set $\Sigma \coloneqq \bigcap_{n \in \N} \Sigma_n \subset A^G$. 
Then $\Sigma$ is a linear subshift and, by our assumptions, it is of finite type. 
It follows from Theorem \ref{t:LSFT-DCC} that the sequence $(\Sigma_n)_{n \in \N}$ eventually stabilizes.
The proof of Corollary \ref{c:DCC-FT} is complete. \hfill $\Box$

\subsection*{Proof of Theorem \ref{t:noether}}
We first observe that if $G$ is uncountable then, on the one hand $G$ is not finitely generated and thus is not Noetherian 
(that is, it does not satisfy the maximal condition on subgroups) and therefore the group algebra $K[G]$ is not one-sided Noetherian 
(cf.\ \cite[Lemma 2.2, Chapter 10]{passman}), and, on the other hand, $G$ is not of $K$-linear Markov type, since the linear subshift consisting of all constant configurations in $K^G$ is not of finite type (cf.\ Corollary \ref{c:constant-conf-SFT}).

Thus, in order to prove Theorem \ref{t:noether}, it is not restrictive to assume that $G$ is countable.

Recall that $\LCA(G,A)$ denotes the $K$-algebra of all linear cellular automata $\tau \colon A^G \to A^G$
(cf.\ \cite[Section 8.1]{book}).

The evaluation map $(\tau, x) \mapsto \tau(x)$, where $\tau \in \LCA(G,A)$ and $x \in A^G$, yields a $K$-bilinear map 
$\LCA(G,A) \times A^G \to A^G$.
\par
Given a subset $\Gamma$ in $\LCA(G,A)$, set
\begin{equation}
\label{e:ideal-perp}
\Gamma^\perp \coloneqq \bigcap_{\tau \in \Gamma} \ker(\tau) \subset A^G.
\end{equation}
\par
Since every map $\tau \in \LCA(G,A)$ is linear, continuous, and $G$-equivariant, we deduce immediately that its kernel
$\ker(\tau)$ is a linear subshift of $A^G$. Moreover, since the set of all linear subshifts in $A^G$ is closed under
intersections, we have that $\Gamma^\perp$ is a linear subshift of $A^G$.
\par
Given a subset $\Sigma \subset A^G$, set
\begin{equation}
\label{e:sigma-perp}
\Sigma^\perp \coloneqq
\{\tau \in \LCA(G,A): \Sigma \subset \ker(\tau)\} \subset \LCA(G,A).
\end{equation}
\par
We claim that $\Sigma^\perp$ is a left ideal in $\LCA(G,A)$. First of all, we clearly have $0 \in \Sigma^\perp$,
since $\Sigma \subset A^G = \ker(0)$. Suppose that $\tau_1, \tau_2 \in \Sigma^\perp$.
Then $(\tau_1-\tau_2)(x) = \tau_1(x) - \tau_2(x) = 0-0 = 0$ for all $x \in \Sigma$, showing that $\tau_1 - \tau_2 \in \Sigma^\perp$.
Finally, if $\tau \in \LCA(G,A)$, we have $(\tau \circ \tau_1)(x) = \tau(\tau_1(x)) = \tau(0) = 0$ for all $x \in \Sigma$, 
showing that $\tau \circ \tau_1 \in \Sigma^\perp$.
This proves the claim.
\par
We note also that if $\Sigma_1, \Sigma_2 \subset A^G$, then 
\begin{equation}
\label{e:inclusions}
\Sigma_1 \subset \Sigma_2 \ \implies \ \Sigma_2^\perp \subset \Sigma_1^\perp.
\end{equation}

We have the following key lemmata:

\begin{lemma}
\label{l:LMT}
Let $G$ be a group, let $A$ be a vector space over a field $K$, and let $\Sigma \subset A^G$ be a linear subshift. Then
\begin{equation}
\label{e:involution}
(\Sigma^\perp)^\perp = \Sigma.
\end{equation}
\end{lemma}

\begin{proof}
It trivially follows from the definitions that $\Sigma \subset (\Sigma^\perp)^\perp$.
In order to show the other inclusion, let $x \in A^G \setminus \Sigma$ and let us show that $x \notin (\Sigma^\perp)^\perp$.
Since $\Sigma$ is closed, by the definition of prodiscrete topology we can find
a finite subset $\Omega \subset G$ such that $x\vert_\Omega \notin \Sigma_\Omega$.
It is a classical and easy argument in Linear Algebra that there exists a linear map
$\mu \colon A^\Omega \to A$ such that $\mu\vert_{\Sigma_\Omega} \equiv 0$,
that is, $\Sigma_\Omega \subset \ker(\mu)$, and $\mu(x\vert_\Omega) \neq 0$.
It is then clear that the linear CA $\tau$ with memory set $\Omega$ and local defining map $\mu$ satisfies that 
$\Sigma \subset \ker(\tau)$, that is, $\tau \in \Sigma^\perp$, but $\tau(x) \neq 0$. 
Thus $x \notin (\Sigma^\perp)^\perp$.
\end{proof}

In the proof of the following lemma, we explicitly use the $K$-algebra isomorphism $\Mat_d(K[G]) \cong \LCA(G,K^d)$ we alluded to above
(cf.\ \cite[Corollary 8.7.8]{book}) for $d=1$. This is given by associating with each $\alpha \in K[G]$ the linear cellular automaton 
$\tau_\alpha \colon K^G \to K^G$
with memory set $M_\alpha \coloneqq \{g \in G: \alpha(g) \neq 0\}$, the support of $\alpha$, and local defining map
$\mu_\alpha \colon K^{M_\alpha} \to K$ defined by setting $\mu_\alpha(y) \coloneqq \sum_{h \in M_\alpha}\alpha(h)y(h)$ for all
$y \in K^{M_\alpha}$. 
\par
We shall also make use of the following notation.
Given $\alpha \in K[G]$, for every finite subset $E \subset G$ such that $M_\alpha \subset E$ we define the linear map 
$\mu_{\alpha, E} \colon K^E \to K$ by setting $\mu_{\alpha, E} \coloneqq \mu_\alpha \circ \pi_{M_\alpha, E}$, 
where $\pi_{M_\alpha, E} \colon K^E \to K^{M_\alpha}$ is the projection map induced by the inclusion $M_\alpha \subset E$. 
Note that $\mu_{\alpha, E}$ is the local defining map of $\tau_{\alpha}$ associated with the memory set $E$. 

\begin{lemma} 
\label{l:markov-hilbert} 
Let $G$ be a countable group and let $K$ be a field. Let $\Gamma \subset K[G]$ be a left ideal. 
Suppose that $\Gamma^\perp \subset K^G$ is a linear SFT. 
Then $\Gamma$ is a finitely generated left ideal. 
\end{lemma} 

\begin{proof} 
Since $G$ is countable, we can find an increasing sequence $(E_n)_{n \in \N}$ of finite subsets of $G$ such that 
$G = \bigcup_{n \in \N} E_n$. For every $n \in  \N$, let $\Gamma_n \subset \Gamma$ be the ideal of $K[G]$ generated 
by the elements of $\Gamma$ whose supports are contained in $E_n$. Then $\Gamma_n \subset \Gamma_{n+1}$ for all $n \in \N$
and $\Gamma = \bigcup_{n \in  \N} \Gamma_n$. We thus obtain a decreasing sequence $(\Gamma_n^\perp)_{n \in \N}$ of 
linear subshifts of $K^G$. 
\par
Remark that we can write 
\[
\Gamma^\perp = \bigcap_{\alpha \in \Gamma} \Ker(\tau_\alpha) = \bigcap_{n \in \N} \bigcap_{\alpha \in \Gamma_n} \Ker(\tau_\alpha) = \bigcap_{n \in \N} \Gamma_n^\perp.  
\]
\par 
Since, by hypothesis, the linear subshift $\Gamma^\perp \subset K^G$ is of finite type, we deduce from Theorem~\ref{t:LSFT-DCC} 
that there exists $n_0  \in  \N$ such that $\Gamma_{n}^\perp = \Gamma_{n_0}^\perp$ for every $n \geq n_0$, equivalently,
$\Gamma^\perp = \Gamma_{n_0}^\perp$. \\
\par 

\noindent
{\bf Claim.} $\Gamma = \Gamma_{n_0}$. 
\begin{proof}[Proof of the claim]
Set $J \coloneqq \Gamma_{n_0}$ and suppose by contradiction that there exists $\alpha \in \Gamma \setminus J$.   
Let $m_0 \in \N$ be such that $E_{m_0}$ contains the support $M_\alpha \subset G$ of $\alpha$. 
\par
For every $m \in \N$, we set $V_m \coloneqq K^{E_m}$ and denote by $V_m^*$ the dual $K$-vector space of $V_m$. 
Given a vector subspace $W_m \subset V_m$ (resp.\ $J_m \subset V_m^*$) we set $W_m^\perp \coloneqq \{v^* \in V_m^*: W_m \subset \ker(v^*)\} \subset V_m^*$ (resp.\ $J_m^\perp \coloneqq \bigcap_{v^* \in J_m} \ker(v^*) \subset V_m$). 
Since $V_m$ is finite-dimensional, we have $(J_m^\perp)^\perp = J_m$.

We then denote by $J_m \subset J$ the subset containing all elements of $J$ whose supports are contained in $E_m$. 
Observe that $J_m \subset J_{m+1}$ and $J = \bigcup_{m \geq m_0} J_m$. 
We regard $J_m$ as a linear subspace of $V_m^*$ via the map $\beta \mapsto \mu_{\beta, E_m}$.
This way, setting $W_m \coloneqq \bigcap_{\beta \in J_m} \Ker(\mu_{\beta, E_m}) \subset V_m$, 
we have $W_m = J_m^\perp$ and therefore
\begin{equation}
\label{e:duality}
\{v^* \in V_m^*: W_m \subset \ker(v^*)\} = W_m^\perp = (J_m^\perp)^\perp = J_m.
\end{equation}

From this we deduce that for every $m \geq m_0$ 
\[
U_m \coloneqq W_m \setminus \Ker(\mu_{\alpha, E_m}) \neq \varnothing.
\]
Indeed, otherwise, we would have $W_m \subset \Ker(\mu_{\alpha, E_m})$ 
so that, by \eqref{e:duality}, $\alpha \in J_m  \subset J$, a contradiction since $\alpha \notin J$. 
\par 
For every $m \geq n \geq m_0$, let $\pi_{nm} \colon K^{E_m} \to K^{E_n}$ be the projection map induced by the inclusion $E_n \subset E_m$. 
It is clear that $\pi_{nm}(U_m) \subset U_n$ since 
$\Ker(\mu_{\alpha, E_m}) = \Ker(\mu_\alpha) \times K^{E_m \setminus M_\alpha} \subset K^{E_m}$ and 
$\pi_{nm}(W_m) \subset W_n$ for all $m \geq n \geq m_0$. 
Therefore, we obtain an inverse system $(U_m)_{m \geq m_0}$ of nonempty sets with transition maps 
$\varphi_{nm} \coloneqq \pi_{nm}\vert_{U_m} \colon U_m \to U_n$ for $m \geq n \geq m_0$. 
\par 
As in Lemma~\ref{l:inverse-limit-closed-im}, an immediate application of the Mittag-Leffler condition to the inverse system 
$(U_m)_{m \geq m_0}$ shows that there exists a configuration $c \in \varprojlim_{m \geq m_0} U_m \subset \varprojlim_{m \geq m_0} W_m$.
Let us show that $c \in J^\perp = \bigcap_{\beta \in J}\ker(\tau_\beta) \subset K^G$.
Let $\beta \in J$ and let $g \in G$. Since $J$ is an ideal of $K[G]$ and
$J = \bigcup_{m \geq m_0} J_m$, there exists $m \geq m_0$ such that $g\beta \in J_m$.
Since $c\vert_{E_m} \in W_m$, it follows from the definition of $W_m$ that
\[
\tau_\beta(c)(g) = \mu_{\beta, E_m}((g^{-1}c)\vert_{E_m})
= \mu_{g\beta, E_m}(c\vert_{E_m}) = 0.
\]
\par
Since $g \in G$ was arbitrary, this shows that $\tau_\beta(c) = 0$. Since $\beta \in J$ was arbitrary, this shows that $c \in J^\perp$.
On the other hand, by construction, we have that $\mu_\alpha(c\vert_{M_\alpha}) \neq 0$ so that $\tau_\alpha(c) \neq 0$. 
Since $\alpha \in J$, we deduce that $c \notin J^\perp$, a contradiction.
The claim is proved. 
\end{proof}

We are now in a position to show that $\Gamma$ is finitely generated as a left ideal. With the above notation,
$J_{n_0}$, the subset consisting of all elements in $J = \Gamma$ whose supports are contained in $E_{n_0}$
is a vector subspace of $V_{n_0} = K^{E_{n_0}}$, and therefore is finite dimensional. 
It is then clear that any vector basis of $J_{n_0}$ also generates $\Gamma_{n_0} = \Gamma$ as a left ideal. 
We conclude that $\Gamma$ is a finitely generated left ideal of $K[G]$.  
\end{proof}

We are now in a position to prove Theorem \ref{t:noether}.

Recall that we assume that $G$ is countable.
Suppose first that the group algebra $K[G]$ is one-sided Noetherian. 
Let $A$ be a finite-dimensional vector space over $K$ and let $d = \dim_K(A)$.
We then observe that since $K[G]$ is one-sided Noetherian, so is the finitely generated left $K[G]$-module $\Mat_d(K[G])$, 
the $K$-algebra of $d \times d$ matrices with coefficients in the group ring $K[G]$.
Since every left ideal in $\Mat_d(K[G])$ is trivially a left $K[G]$-module, we deduce that $\Mat_d(K[G])$ is
one-sided Noetherian as well as a ring.
As mentioned above (cf.\ \cite[Corollary 8.7.8]{book}), once fixed a vector basis for $A$, 
there exists a canonical $K$-algebra isomorphism of $\LCA(G,A)$ onto $\Mat_d(K[G])$.
We deduce that $\LCA(G,A)$ is one-sided Noetherian.
\par
In order to show that $G$ is of $K$-linear Markov type, let $(\Sigma_n)_{n \in \N}$ be a decreasing sequence of linear subshifts in $A^G$
and let us show that it stabilizes.  
Setting $\Gamma_n \coloneqq \Sigma_n^\perp$ for all $n \in \N$, we get an increasing sequence $(\Gamma_n)_{n \in \N}$ of
left ideals in $\LCA(G,A)$. 
Since the latter is left-Noetherian, such a sequence stabilizes, that is, there
exists $n_0 \in \N$ such that $\Gamma_n = \Gamma_{n_0}$ for all $n \geq n_0$. It then follows from Lemma \ref{l:LMT} that
$\Sigma_n = \Gamma_n^\perp =  \Gamma_{n_0}^\perp = \Sigma_{n_0}$ for all $n \geq n_0$, that is, $(\Sigma_n)_{n \in \N}$
stabilizes. This shows that $G$ is of $K$-linear Markov type.
 \par 
Conversely, suppose that $G$ is of $K$-linear Markov type and let $\Gamma \subset K[G]$ be a left ideal. 
Then the linear subshift $\Gamma^\perp \subset K^G$ is of finite type. 
Lemma~\ref{l:markov-hilbert} implies that $\Gamma$ is finitely generated. 
This shows that the group algebra $K[G]$ is one-sided Noetherian. 

The proof of Theorem \ref{t:noether} is complete. \hfill $\Box$

\subsection*{Proof of Corollary \ref{c:LMT}}
Let $G$ be a polycyclic-by-finite group and let $K$ be a field.
It follows from a famous result of P.\ Hall \cite{hall} (see also \cite[Corollary 10.2.8]{passman}) 
that $K[G]$ is one-sided Noetherian. We then deduce from Theorem \ref{t:noether} that $G$ is of $K$-linear Markov type. \hfill $\Box$
\vspace{0.3cm}

\noindent
{\it Remark.}
(1) At our knowledge, it is not known whether or not there exist groups $G$, other than the polycyclic-by-finite groups, 
whose group algebra $K[G]$ is one-sided Noetherian. See Section \ref{ss:GLMT} for more on this.
\par
\noindent
(2) An alternative and self-contained proof of Corollary \ref{c:LMT} is obtained from Lemma \ref{l:finite} and Lemma
\ref{l:cyclic} below combined with an easy induction argument. For the details see Remark \ref{r:alternative}.

\subsection*{Proof of Theorem \ref{t:limit-set}} 
(i) It follows from Theorem \ref{t:closed-image} that $\tau^n(\Sigma)$ is a linear subshift in $A^G$ for all $n \in \N$.
Since the intersection of any family of linear subshifts is itself a linear subshift, we deduce that 
$\Omega(\tau) = \bigcap_{n \in \N} \tau^n(\Sigma)$ is a linear subshift. 
\par
(ii) Let $x \in \Omega(\tau)$, that is, $x \in \tau^n(\Sigma)$ for every $n \geq 0$.  
Thus $\tau(x) \in \tau^{n+1}(\Sigma)$ for every $n \geq 0$ and it follows that $\tau(x) \in \Omega(\tau)$. 
Therefore, $\tau(\Omega(\tau)) \subset \Omega(\tau)$. 
For the converse inclusion, let $y \in \Omega(\tau)$. Let also $M \subset G$ be a memory set for $\tau$ such that
$1_G \in M$ and $M = M^{-1}$. Then, by Proposition~\ref{p:phi-surg}, 
there exists $x \in \varprojlim_{(\Omega,n) \in I} \Sigma_{\Omega,n}$ 
such that $\Phi(x)=y$. 
On the other hand, \eqref{e:backward-orbits} tells us that $\Phi^{-1}(y) \subset 
\varprojlim_{(\Omega,n) \in I} \Sigma_{\Omega,n}$ is the set of backward orbits of $y$ under $\tau$. 
Hence, we can find $z \in \Omega(\tau)$ such that $\tau(z)= y$. Thus,  
$\Omega(\tau) \subset \tau(\Omega(\tau))$ and equality follows. 
\par
(iii) As already mentioned in the Introduction, the inclusions $\Per(\tau) \subset \Rec(\tau) \subset \NW(\tau)$
are immediate from the definitions.
In \cite[Proposition 2.2]{ccp-2020} it is shown that if
$X$ is a uniform space and $f \colon X \to X$ is a continuous map, then $\NW(f) \subset \Crec(f)$.
Since every cellular automaton is continuous, we deuce that $\NW(\tau) \subset \Crec(\tau)$.
In \cite[Proposition 2.3]{ccp-2020} it is shown that if
$X$ is a Hausdorff uniform space and $f \colon X \to X$ is a uniformly continuous map such that
$f^n(X)$ is closed in $X$ for all $n \in \N$, then $\Crec(f) \subset \Omega(f)$.
In our setting, uniform continuity of $\tau$ is a general property of cellular automata 
already mentioned in the Introduction. Moreover, $\tau^n(\Sigma)$ is closed in $\Sigma$ for all $n \in \N$ by
Theorem \ref{t:closed-image}. 
We deduce the last inclusion, namely $\Crec(\tau) \subset \Omega(\tau)$.
\par
(iv)
Suppose that $\Omega(\tau)$ is of finite type. It follows from Theorem \ref{t:LSFT-DCC} that the sequence
$(\tau^n(\Sigma))_{n \in \N}$ eventually stabilizes, that is, there exists $n_0 \geq 1$ such that $\tau^n(\Sigma) = \tau^{n_0}(\Sigma)$
for all $n \geq n_0$. This shows that $\tau$ is stable.
\par
(v) Suppose that $\Omega(\tau)$ is finite-dimensional. It follows from Proposition \ref{p:FD-SFT} that $\Omega(\tau)$ is of finite type.
Using (iv) we deduce that $\tau$ is stable.

This ends the proof of Theorem \ref{t:limit-set}. \hfill $\Box$

\subsection*{Proof of Corollary \ref{c:limit-set}}
We only need to prove the statements for $G$ not finitely generated.
Let $M \subset G$ be a finite subset serving as a memory set for both $\Sigma$ and $\tau$, and denote by 
$H \subset G$ the subgroup generated by $M$.
\par
The proof of Theorem \ref{t:limit-set}.(ii) did not use any finite generation assumption on $G$ and therefore holds true 
in the present setting as well.
\par
(i) It follows from Theorem \ref{t:limit-set}.(i) applied to the restriction cellular automaton $\tau_H \colon \Sigma_H \to \Sigma_H$ that
$\Omega(\tau_H)$ is a linear subshift. As a consequence, $\Omega(\tau_c)$ are closed in $A^c$ for all $c \in G/H$. As products of closed
subspaces are closed in the product topology, we deduce from Lemma \ref{l:restriction-ls}.(i) that 
$\Omega(\tau) = \prod_{c \in G/H} \Omega(\tau_c)$ is also closed in $A^G$. Since $\Omega(\tau)$ is a $K$-linear and $G$-invariant
subset of $A^G$, we conclude that it is a linear subshift of $A^G$.
\par
(ii) It follows from Theorem \ref{t:limit-set}.(ii) applied to the restriction cellular automaton $\tau_H \colon \Sigma_H \to \Sigma_H$
that $\tau_H(\Omega(\tau_H)) = \Omega(\tau_H)$. As a consequence, $\tau_c(\Omega(\tau_c)) = \Omega(\tau_c)$ for all $c \in G/H$. 
We deduce from Lemma \ref{l:restriction-ls}.(i) that 
$\tau(\Omega(\tau)) = \prod_{c \in G/H} \tau_c(\Omega(\tau_c)) = \prod_{c \in G/H} \Omega(\tau_c) = \Omega(\tau)$.
\par
(iii) Just note that, by virtue of Theorem \ref{t:closed-image}, $\tau^n_H(\Sigma_H)$ is closed in $A^H$ for all $n \in \N$.
Hence, by Lemma \ref{l:restriction-cip}, $\tau^n(\Sigma)$ is closed in $A^G$ for all $n \in \N$, and
the proof of Theorem \ref{t:limit-set}.(iii) applies verbatim.
\par
(iv) Up to enlarging $M \subset G$, if necessary, we may suppose that $M$ also serves as a memory set for the SFT $\Omega(\tau)$,
say $\Omega(\tau) = \Sigma(M, W) \subset A^G$ for some $W \subset A^M$.
We have that $\Omega(\tau)_H = \Omega(\tau_H) = \Sigma(M, W) \subset A^H$ is of finite type as well.
It then follows from Theorem \ref{t:limit-set}.(iv) applied to the restriction cellular automaton $\tau_H$, that $\tau_H$ is stable.
Since stability is invariant under the operation of restriction, we deduce that $\tau$ is itself stable.
\par
(v) If $\Omega(\tau)$ is finite-dimensional, so is $\Omega(\tau_H) = \Omega(\tau)_H$. 
It then follows from Theorem \ref{t:limit-set}.(v) applied to the restriction cellular automaton $\tau_H$, that $\tau_H$ is stable.
Thus $\tau$ is itself stable. 

This proves Corollary \ref{c:limit-set}. \hfill $\Box$

\subsection*{Proof of Corollary \ref{c:other-characterization}}
We first observe that every polycyclic-by-finite group is amenable (see, for instance \cite[Chapter 4]{book}).
Let $\Lambda \subset A^G$ be a strongly irreducible subshift such that $\Lambda \subset \Sigma$, $\tau(\Lambda) \subset \Lambda$, and
such that the restriction linear CA $\tau\vert_\Lambda \colon \Lambda \to \Lambda$ is pre-injective.
Since $G$ is polycyclic-by-finite, Corollary~\ref{c:LMT} ensures that $\Lambda$ is a linear subshift of finite type.
Since $G$ is amenable, the implication pre-injectivity $\implies$ surjectivity in the Garden of Eden theorem 
\cite[Theorem~1.2]{cc-goe-sft} yields the equality $\tau(\Lambda) = \Lambda$. 
It follows immediately that $\Lambda \subset \Omega(\tau)$.
\par
Theorem~\ref{t:limit-set}.(i) and Corollary~\ref{c:LMT} imply that $\Omega(\tau)$ is a linear subshift of finite type.
Thus, by Theorem~\ref{t:limit-set}.(iv), $\tau$ is stable and therefore 
there exists an integer $n \geq 1$ such that $\tau^n(\Sigma) = \Omega(\tau)$. 
Since the image of a strongly irreducible subshift under a CA is also strongly irreducible, it follows that
$\Omega(\tau)$ is a strongly irreducible linear SFT.
By Theorem~\ref{t:limit-set}.(ii), $\tau(\Omega(\tau)) \subset \Omega(\tau)$ and the restriction linear CA
$\tau\vert_{\Omega(\tau)} \colon \Omega(\tau) \to \Omega(\tau)$ is surjective.
We can thus conclude from the implication surjectivity $\implies$ pre-injectivity in the Garden of Eden theorem 
\cite[Theorem~1.2]{cc-goe-sft} that $\tau\vert_{\Omega(\tau)}$ is pre-injective. 

The proof of Corollary \ref{c:other-characterization} is complete.  \hfill $\Box$

\subsection*{Proof of Theorem \ref{t:char-nilpotent-finite-lca-one}}
Let $H \subset G$ be a finitely generated subgroup containing both a memory set for $\Sigma$ and a memory set for $\tau$.
By virtue of Lemma \ref{l:restriction-ls}, we have, on the one hand that $\tau$ is nilpotent if and only if $\tau_H$ is, and, one the other hand that $\Omega(\tau) = \{0\}$ if and only if $\Omega(\tau_H) = \{0\}$. 
Thus, it is not restrictive to suppose that $G = H$ is finitely generated.
\par
Suppose that $\tau$ is nilpotent. Then there exists $n_0 \geq 1$ such that $\tau^{n_0}(\Sigma) = \{0\}$. 
It then follows that $\Omega(\tau) = \{0\}$.
\par
Conversely, suppose (b). Then $\Omega(\tau)$ is of finite type. By the characterization of linear SFT in Theorem \ref{t:LSFT-DCC},
the sequence $(\tau^n(\Sigma))_{n \in \N}$ eventually stabilizes, that is, there exits $n_0 \geq 1$ 
such that $\tau^{n_0}(\Sigma) = \Omega(\tau) = \{0\}$. This shows that $\tau$ is nilpotent.

This completes the proof of Theorem \ref{t:char-nilpotent-finite-lca-one}. \hfill $\Box$

\subsection*{Proof of Theorem \ref{t:char-nilpotent-finite-lca}}
We shall prove the implications 
\[
\mbox{(a) $\iff$ (b)  \ \ \ \ and \ \ \ \ (a) $\implies$ (c) $\implies$ (d) $\implies$ (e) $\implies$ (a).}
\]
The implication (a) $\implies$ (b) is trivial.
\par
Suppose that $\tau$ is pointwise nilpotent, so that for every $x \in \Sigma$, 
there exists an integer $n_x \geq 1$ such that $\tau^n(x) = 0$ for all $n \geq n_x$.
Since $G$ is finitely generated, it is countable.
Then, the configuration space $A^G$, being a countable product of discrete (and therefore completely metrizable) spaces, it admits a complete metric compatible with its topology, and hence is a Baire space. Since $\Sigma$ is closed in $A^G$, it is a Baire space as well. 
For each integer $n \geq 1$, the set
\[ 
X_n \coloneqq (\tau^n)^{-1}(0) = \{x \in \Sigma : \tau^n(x) = 0\}  
\]
is a linear subshift of $A^G$ contained in $\Sigma$. We have $\Sigma = \bigcup_{n \geq 1} X_n$ by our hypothesis on $\tau$.
By the Baire category theorem, there is an integer $n_0 \geq 1$ such that $X_{n_0}$ has a nonempty interior.
Since $\Sigma$ is topologically mixing and $G$ is infinite, $\Sigma$ is topologically transitive.
It follows from a standard fact (cf.\ \cite[Lemma A.3]{ccp-2020}) that $X_{n_0} = \Sigma$, equivalently, $\tau^{n_0}(\Sigma) = \{0\}$. 
The latter is equivalent to $\tau$ being nilpotent, and the implication (b) $\implies$ (a) follows. 
From the first implication we deduce that in fact (a) $\iff$ (b).
\par
The implication (a) $\implies$ (c) is obvious.
\par 
The implication (c) $\implies$ (d) is clear since $\Omega(\tau)$ is a vector subspace of $\tau^{n_0}(\Sigma)$.
\par
The implication (d) $\implies$ (e) follows from Proposition \ref{p:TopTransLinShifts} since any topologically mixing action of an infinite group is topologically transitive.
\par
Finally, suppose (e). As $\Omega(\tau)$ is of finite type, we deduce from Theorem \ref{t:LSFT-DCC}
that the sequence $(\tau^n(\Sigma))_{n \in \N}$ eventually stabilizes, that is, there exists an $n_0 \in \N$ such that 
$\tau^{n_0}(\Sigma) = \Omega(\tau) = \{0\}$. Thus $\tau$ is nilpotent. This shows the outstanding implication 
(e) $\implies$ (a), and the proof  of Theorem \ref{t:char-nilpotent-finite-lca} is complete. \hfill $\Box$

\subsection*{Proof of Corollary \ref{c:char-nilpotent-finite-lca}}
We only need to prove the equivalences for $G$ not finitely generated. Let $H \subset G$ be a finitely generated subgroup containing
both a memory set for $\Sigma$ and a memory set for $\tau$. Observe that $[G:H] = \infty$, since $G$ is not finitely generated.
Denote by  $\tau_H \colon \Sigma_H \to \Sigma_H$ the corresponding restriction cellular automaton.
\par
The implication (a) $\implies$ (b) is trivial.
\par
Suppose (b). It is straightforward that $\tau_H$ is also pointwise nilpotent. It then follows from the finitely generated case (i.e.,
from the implication (b) $\implies$ (a) in Theorem \ref{t:char-nilpotent-finite-lca}) that $\tau_H$ is nilpotent. We then deduce from Lemma \ref{l:restriction-ls}.(iii) that
$\tau$ is itself nilpotent. This shows the implication (b) $\implies$ (a). Combined with the previous implication, this gives the equivalence
(a) $\iff$ (b).
\par
The implications (a) $\implies$ (c) $\implies$ (d) are trivial.
\par
Suppose (d). Recalling that $H$ has infinite index in $G$, we deduce from Lemma \ref{l:restriction-ls}.(i) 
that $\Omega(\tau_H) = \{0\}$ and $\Omega(\tau) = \{0\}$.
This shows the implication (d) $\implies$ (e).
\par
The final implication (e) $\implies$ (a) follows from Theorem \ref{t:char-nilpotent-finite-lca-one}.

The proof of Corollary \ref{c:char-nilpotent-finite-lca} is complete. \hfill $\Box$

\section{Groups of $K$-linear Markov type}
\label{ss:GLMT}
We have seen in Corollary \ref{c:constant-conf-SFT} that the condition that $G$ be finitely generated cannot be removed from the assumptions in Proposition \ref{p:FD-SFT}. 
More generally, if a finitely generated group $G$ admits a subgroup $H$ which is not finitely generated (for instance, if $G$
contains a subgroup $K$ isomorphic to $F_2$, the free group of rank $2$, and $H = [K,K] \subset K$ its commutator subgroup) 
then the subshift consisting of all configurations  $x \in A^G$ which are constant on each left coset of $H$ in $G$ is not 
of finite type (note that $H$ has necessarily infinite index in $G$).
\par
Recall that a group $G$ satisfies the \emph{maximal condition on subgroups} if any ascending sequence $G_0 \subset G_1 \subset \cdots
\subset G_n \subset G_{n+1} \subset \cdots \subset G$ of subgroups eventually stabilizes, that is, there exists $n_0 \geq 1$ such that
$G_n = G_{n_0}$ for all $n \geq n_0$. 
It is immediately verified that a group $G$ satisfies the maximal condition on subgroups if and
only if all of its subgroups are finitely generated.
A group satisfying the maximal condition on subgroups is also called a \emph{Noetherian group}.
From the above discussion we immediately deduce the following.

\begin{corollary}
\label{c:noetherian-group}
Let $G$ be a group of $K$-linear Markov type for some field $K$. Then $G$ is Noetherian. 
In particular, $G$ is finitely generated.
\end{corollary}

As remarked above, we don't know whether or not the class of polycyclic-by-finite groups coincides with the class of groups of $K$-linear Markov type.
We remark that there exist Noetherian groups constructed by A.Y.\ Olshanskii \cite{olshanskii}, 
for which the group algebra is not known to be one-sided Noetherian, equivalently (cf.\ Theorem\ref{t:noether}), 
it is not known whether or not they are of $K$-linear Markov type.

On the other hand, it follows from the work of L.\ Bartholdi \cite{bartholdi} and P.\ Kropholler and K.\ Lorensen \cite{KL}, 
and Theorem \ref{t:noether}, that if $G$ is of $K$-linear Markov type, then $G$ is necessarily amenable. 

We refer to Mathoverflow \cite{MO} for other interesting information.

In the next two lemmas we show that the class of groups of $K$-linear Markov type is closed under finite and cyclic extensions.

\begin{lemma}
\label{l:finite}
Let $G$ be a countable group and let $H \subset G$ be a normal subgroup of finite index. 
Suppose that $H$ is of linear Markov type. Then also $G$ is of linear Markov type.
\end{lemma}
\begin{proof}
Let $A$ be a finite-dimensional vector space and let $\Sigma \subset A^G$ be a linear subshift.
Let $T \subset G$ be a complete set of representatives for the cosets of $H$ in $G$, so that $G = HT$.
Then $B \coloneqq A^T$ is a finite-dimensional vector space and the map $\varphi \colon A^G \to B^H$ defined by
\begin{equation}
\label{e:def-phi}
\left(\varphi(x)(h)\right)(t) = x(ht)
\end{equation}
for all $x \in A^G$, $h \in H$, and $t \in T$, is a linear isomorphims and uniform homeomorphism. Moreover, 
\[
\left(\varphi(kx)(h)\right)(t) = (kx)(ht) = x(k^{-1}ht) = \left(\varphi(x)(k^{-1}h)\right)(t) = \left(k\varphi(x)(h)\right)(t)
\]
for all $x \in A^G$, $k,h \in H$, and $t \in T$, showing that $\varphi$ is $H$-equivariant.
\par
Then $\Sigma' \coloneqq \varphi(\Sigma)$ is a linear subshift in $B^H$. 
Since $H$ is of linear Markov type, and $B$ is finite-dimensional, there exists a finite subset $D_H \subset H$ 
and a subspace $P_H \subset B^{D_H}$ such that $\Sigma' = \Sigma(B^H;D_H,P_H)$. 
Set $D \coloneqq D_HT \subset G$ and consider the map $\psi \colon A^D \to B^{D_H}$
defined by
\begin{equation}
\label{e:def-psi}
\left(\psi(y)(h)\right)(t) = y(ht)
\end{equation}
for all $y \in A^D$, $h \in D_H$, and $t \in T$. Then $\psi$ is a linear isomorphism and a uniform homeomorphism. 
Let us set $P \coloneqq \psi^{-1}(P_H) \subset A^D$.
Note that if $x \in A^G$, $h \in D_H$, and $t \in T$ we have
\[
\left(\varphi(x)\vert_{D_H}(h)\right)(t) = \left(\varphi(x)(h)\right)(t) = x(ht) = x\vert_D(ht) = \left(\psi(x\vert_D)(h)\right)(t)
\]
so that
\begin{equation}
\label{e:phi-psi}
\varphi(x)\vert_{D_H} = \psi(x\vert_D).
\end{equation}
It follows that
\begin{align*}
x \in \Sigma & \iff tx \in \Sigma, \mbox{ for all } t \in T && \mbox{(by $G$-invariance of $\Sigma$)}\\
& \iff \varphi(tx) \in \Sigma', \mbox{ for all } t \in T && \mbox{(by definition of $\Sigma'$ and $\varphi$ being $1$-$1$)}\\
& \iff \left(h \varphi(tx)\right)\vert_{D_H} \in P_H, \mbox{ for all } h \in H \mbox{ and } t \in T && \mbox{(since $\Sigma' = \Sigma(B^H;D_H,P_H)$)}\\
& \iff \left(\varphi(htx)\right)\vert_{D_H} \in P_H, \mbox{ for all } h \in H \mbox{ and } t \in T && 
\mbox{(by $H$-equivariance of $\varphi$)}\\
& \iff \left(\varphi(gx)\right)\vert_{D_H} \in P_H, \mbox{ for all } g \in G && \mbox{(since $G = HT$)}\\
& \iff \psi^{-1}\left(\left(\varphi(gx)\right)\vert_{D_H}\right) \in P, \mbox{ for all } g \in G && \mbox{(by definition of $P$ and $\psi$ being $1$-$1$)}\\
& \iff (gx)\vert_D \in P, \mbox{ for all } g \in G && \mbox{(by \eqref{e:phi-psi})}\\
& \iff x \in \Sigma(A^G; D, P). &&
\end{align*}
This shows that $\Sigma = \Sigma(A^G; D, P)$ is of finite type.
We deduce that $G$ is of linear Markov type.
\end{proof}

The following is the linear (and therefore simpler) version of the more general result \cite[Theorem 7.2]{phung-2020}. 

\begin{lemma}
\label{l:cyclic}
Let $G$ be a countable group and let $H \subset G$ be a normal subgroup such that $G/H$ is infinite cyclic. 
Suppose that $H$ is of linear Markov type. Then also $G$ is of linear Markov type.
\end{lemma}
\begin{proof}
Let $A$ be a finite-dimensional vector space and let $\Sigma \subset A^G$ be a linear subshift.
Let $a \in G$ such that $aH$ generates $G/H \cong \Z$ and set $T' \coloneqq \{a^n: n \in \Z\}$. Then
$T'$ is a complete set of representatives for the cosets of $H$ in $G$ so that $G = HT'$. 
Since $H$ is also countable, we can find an increasing sequence $(F_m)_{m \in \N}$ of finite subsets $F_m \subset H$
such that $1_H \in F_0$ and $H = \bigcup_{m \in \N} F_m$. 
For $i,j \in \Z \cup \{-\infty, +\infty\}$ and $i \leq j$ let us set $T_i^j \coloneqq \{a^i, a^{i+1}, \cdots, a^j\} \subset T'$.\\

\noindent
{\bf Claim 1.} {\it For every $n \geq 1$ the set}
\begin{equation}
\label{e:X-n}
X_n \coloneqq \{x\vert_H: x \in \Sigma \mbox{ such that } x(g) = 0_A \mbox{ for all } g \in HT_{-n}^{-1}\} \subset A^H
\end{equation}
{\it is a linear subshift in $A^H$.}\\
\begin{proof}[Proof of the Claim]
Let $n \geq 1$. The fact that $X_n$ is a vector subspace of $A^H$ is clear.
Let now $x \in X_n$. 
Then, there exists $y \in \Sigma$ such that $x = y\vert_H$ and $y(g) = 0_A$ for all $g \in HT_{-n}^{-1}$.
Given $h \in H$, we have $hy \in \Sigma$, because $\Sigma$ is a subshift in $A^G$, and
$(hy)(g) = y(h^{-1}g) = 0_A$ for all $g \in HT_{-n}^{-1}$, since $h^{-1}g \in hHT_{-n}^{-1} = HT_{-n}^{-1}$.
It follows that $hx = (hy)\vert_H \in X_n$, and this shows that $X_n$ is $H$-invariant.
We are only left to show that $X_n$ is closed with respect to the prodiscrete topology in $A^H$.
For $k \geq n$ let us set
\[
X_{n,k} \coloneqq \{x\vert_{F_kT_{-k}^k}: x \in \Sigma \mbox{ such that } x(g) = 0_A \mbox{ for all } g \in HT_{-n}^{-1}\} \subset \Sigma_{F_kT_{-k}^k}.
\]
Note that $X_{n,k}$ is a finite-dimensional vector space. 
For $m \geq k \geq n$, let $\pi_{k,m} \colon A^{F_mT_{-m}^m} \to A^{F_kT_{-k}^k}$ denote the projection map.
Note that if $x \in A^{F_mT_{-m}^m}$ satisfies that $x(g) = 0_A$ for all $g \in F_mT_{-n}^{-1}$, 
then $\pi_{k,m}(x)(g') = 0_A$ for all $g' \in F_kT_{-n}^{-1}$. Hence, setting $p_{k,m} \coloneqq \pi_{k,m}\vert_{X_{n,m}}$ we have
$p_{k,m} \colon X_{n,m} \to X_{n,k}$, and $(X_{n,k}, p_{k,m})_{m \geq k \geq n}$ is an inverse system of finite-dimensional vector spaces.
\par
Let $z \in A^H$ be a configuration belonging to the closure of $X_n$ in $A^H$. We must show that $z \in X_n$.
By definition, for each $k \geq n$ there exists a configuration $x_k \in \Sigma$ such that
\[
x_k\vert_{F_k} = z\vert_{F_k} \mbox{ \ and \ } x_k(g) = 0_A \mbox{ for all } g \in HT_{-n}^{-1}.
\]
Let us set
\[
X_{n,k}(z) \coloneqq \{x\vert_{F_kT_{-k}^k}: x \in \Sigma \mbox{ such that } x\vert_{F_k} = z\vert_{F_k} \mbox{ and } x(g) = 0_A \mbox{ for all } g \in HT_{-n}^{-1}\} \subset X_{n,k}.
\]
Note that $X_{n,k}(z)$ is an affine subset in $A^{F_kT_{-k}^{k}}$. Moreover, for $i \leq j$ we have that $p_{ij}(X_{n,j}(z)) \subset
X_{n,i}(z)$, showing that $(X_{n,k}(z))$ is an inverse system (in fact, an inverse subsystem of $(\Sigma_{F_kT_{-k}^{k}})$).
By Lemma \ref{l:inverse-limit-closed-im}, there exists 
$x \in \varprojlim_{k \geq n} X_{n,k}(z) \subset \varprojlim_{k \geq n} \Sigma_{F_kT_{-k}^{k}} = \Sigma$. 
By construction, we have
\[
x(g) = 0_A \mbox{ for all } g \in F_kT_{-n}^{-1} \mbox{ \ and \ } x\vert_{F_k} = z\vert_{F_k}, \mbox{ for all } k \geq n,
\]
so that, letting $k \to \infty$,
\[
x(g) = 0_A \mbox{ for all } g \in HT_{-n}^{-1} \mbox{ \ and \ } x\vert_H = z.
\]
This shows that $z = x\vert_H \in X_n$. The claim follows.
\end{proof}

It is clear that $X_n \supset X_{n+1}$ for all $n \geq 1$. Thus, as a consequence of Claim 1, 
$(X_n)_{n \in \N}$ is a decreasing sequence of linear subshifts of $A^H$. 
Since $A$ is finite-dimensional and $H$ is of linear Markov type, by Corollary \ref{c:DCC-FT}
the above sequence must stabilize: there exists $n_0 \in \N$ such that $X_n = X_{n_0}$ for all $n \geq n_0$.
Thus, setting 
\[
X \coloneqq \bigcap_{n \in \N} X_n = \{x\vert_H: x \in \Sigma \mbox{ such that } x(g) = 0_A \mbox{ for all } g \in HT_{-\infty}^{-1}\} \subset A^H
\]
we have that $X$ is a linear subshift in $A^H$ and, moreover,
\begin{equation}
\label{e:X-X-n-0}
X = X_{n_0} = \{x\vert_H: x \in \Sigma \mbox{ such that } x(g) = 0_A \mbox{ for all } g \in HT_{-n_0}^{-1}\}.
\end{equation}

Consider the finite set $T \coloneqq T_{-n_0}^{n_0} \subset T'$ and set $\Omega \coloneqq HT \subset G$. 
The action of $H$ on $\Omega$ by left multiplication induces an action of $H$ on $A^\Omega$: this is given by
setting $(hx)(kt) \coloneqq x(h^{-1}kt)$ for all $h,k \in H$, $x \in A^\Omega$, and $t \in T$.\\

\noindent
{\bf Claim 2.} {\it The subset $\Sigma_\Omega \subset A^\Omega$ is $H$-invariant and closed with respect to the 
prodiscrete topology on $A^\Omega$.}\\

\begin{proof}[Proof of the Claim]
Let $z \in \Sigma_\Omega$. Then there exists $x \in \Sigma$ such that $z = x\vert_\Omega$. Given $h \in H$, 
we have, for all $k \in H$ and $t \in T$,
\[
(hz)(kt) = z(h^{-1}kt) = x(h^{-1}kt) = (hx)(kt) =  (hx)\vert_\Omega (kt).
\]
Since $hx \in \Sigma$, we deduce that $hz = (hx)\vert_\Omega \in \Sigma_\Omega$. This shows that $\Sigma_\Omega$ is $H$-invariant.
\par
Since $G$ is countable, we can find an increasing sequence $(E_n)_{n \in \N}$ of finite subsets of $G$ 
such that $G = \bigcup_{n \in \N} E_n$. Setting $F_n \coloneqq E_n \cap \Omega$ for all $n \in \N$, 
we obtain an increasing sequence  $(F_n)_{n \in \N}$ of finite subsets of $\Omega$  such that $\Omega = \bigcup_{n \in \N} F_n$.
Let $d \in A^\Omega$ and suppose it belongs to the closure of $\Sigma_\Omega$ in $A^\Omega$. We must show that $d \in \Sigma_\Omega$.
For each $n \in \N$ there exists $y \in \Sigma_\Omega$ such that $d\vert_{F_n} = y\vert_{F_n}$. Since $y \in \Sigma_\Omega$, there
exists $x \in \Sigma$ such that $y = x\vert_\Omega$. Setting $z \coloneqq x\vert_{E_n} \in \Sigma_{E_n}$, 
we have $z\vert_{F_n} = x\vert_{F_n} = (x\vert_\Omega)\vert_{F_n} = y\vert_{F_n} = d\vert_{F_n}$, so that
the finite-dimensional affine set
\[
Z_n \coloneqq \{z \in \Sigma_{E_n}: z\vert_{F_n} = d\vert_{F_n}\} \subset \Sigma_{E_n}
\]
is nonempty.
It is clear that for $m,n \in \N$ with $m \geq n$ the restriction map $\pi_{n,m} \colon A^{E_m} \to A^{E_n}$ induces, by
restriction, a well defined linear map $p_{n m} \colon Z_m \to Z_n$. Hence, by applying Lemma \ref{l:inverse-limit-closed-im} 
to the inverse system $(Z_m, p_{n m})$,  there exists 
$x \in \varprojlim_{n \to \infty} Z_n \subset \varprojlim_{n \to \infty} \Sigma_{E_n} = \Sigma$.
By definition, we have $x\vert_{F_n} = d\vert_{F_n}$ for every $n \in \N$, so that $x\vert_\Omega = d$.
This shows that $d = x\vert_\Omega \in \Sigma_\Omega$. We deduce that $\Sigma_\Omega$ is closed, and the claim follows.
\end{proof}

\noindent
{\bf Claim 3.} $\Sigma = \Sigma(A^G;\Omega, \Sigma_\Omega)$.
\begin{proof}[Proof of the Claim]
Let us set $\widetilde{\Sigma} \coloneqq \Sigma(A^G;\Omega, \Sigma_\Omega) \subset A^G$. 
It is clear that $\Sigma \subset \widetilde{\Sigma}$. 
To prove the converse inclusion, let $y \in \widetilde{\Sigma}$. 
Then, there exists $z_0 \in \Sigma$ such that $z_0\vert_\Omega = y\vert_\Omega$.
Since also $a^{-1}y \in \widetilde{\Sigma}$, there exists $y_0 \in \Sigma$ such that
$y_0\vert_\Omega = (a^{-1}y)\vert_\Omega$. As a consequence, setting $z_1 \coloneqq ay_0 \in \Sigma$, one has
$z_1(a \omega) = (a^{-1}z_1)(\omega) = y_0(\omega) =  (a^{-1}y)(\omega) = y(a\omega)$ for all $\omega \in \Omega$,
equivalently, $z_1\vert_{a\Omega} = y\vert_{a\Omega}$. Note that $a\Omega = aHT = HaT = HaT_{-n_0}^{n_0} = HT_{-n_0+1}^{n_0+1}$
so that $\Omega \cap (a\Omega) = HT_{-n_0+1}^{n_0} \supset HT_{1}^{n_0}$. 
Thus, for the configuration $z \coloneqq z_0 - z_1 \in \Sigma$  we have $z(g) = 0_A$ for all $g \in HT_{1}^{n_0}$.
Moreover, if $g' \in HT_{-n_0}^{-1}$, then $g \coloneqq a^{n_0+1}g' \in HT_{1}^{n_0}$ and therefore
$(a^{-n_0-1}z)(g') = z(a^{n_0+1}g') = z(g) = 0_A$. As a consequence, 
the configuration $v \coloneqq (a^{-n_0-1}z)\vert_H \in A^H$ is in $X$ (cf.\ \eqref{e:X-X-n-0}).
Set 
\[
L(v) \coloneqq \{x \in \Sigma: x(g) = 0_A \mbox{ for all } g \in HT_{-\infty}^{-1} \mbox{ and } x\vert_H = v\} \subset A^G.
\]
Clearly, $L(v)$ is a nonempty affine subspace: keeping in mind \eqref{e:X-X-n-0}, 
there exists $z' \in \Sigma$ with $z'\vert_H = (a^{-n_0-1}z)\vert_H = v$ and
$z'(g) = 0_A$ for all $g \in HT_{-\infty}^{-1}$, so that $z' \in L(v)$.
Let $c \in L(v)$ and consider the configuration $x \coloneqq z_0 - a^{n_0+1}c \in \Sigma$. 
Let $h \in H$ and set $h' \coloneqq a^{-n_0-1}ha^{n_0+1} \in H$. Then, for $1 \leq n \leq n_0$ we have
\begin{equation}
\label{e:cas-n}
x(ha^n) = z_0(ha^n) - (a^{n_0+1}c)(ha^n) = z_0(ha^n) - c(h'a^{n-n_0-1}) = z_0(ha_n) = y(ha^n),
\end{equation}
where we used the fact that, on the one hand $-n_0 \leq n-n_0-1 \leq -1$ and $c\vert_{HT_{-n_0}^{-1}} = 0$, and,
on the other hand, $HT_{-n_0}^{-1} \subset HT_{-n_0}^{n_0} = \Omega$ and $z_0\vert_{\Omega} = y\vert_\Omega$.
Moreover,
\begin{align*}
x(ha^{n_0+1}) & = x(a^{n_0+1}h') = z_0(a^{n_0+1}h') - (a^{n_0+1}c)(a^{n_0+1}h') && \\
& = z_0(a^{n_0+1}h') - c(h') = z_0(a^{n_0+1}h') - (a^{-n_0-1}z)(h') && \mbox{(since $c \in L(v)$)}\\
& = z_0(a^{n_0+1}h') - z(a^{n_0+1}h') = z_1(a^{n_0+1}h')  && \mbox{(since $z =  z_0 - z_1$)}\\
& = z_1(a(a^{n_0}h')) = y(a(a^{n_0}h')) && \mbox{(since $z_1\vert_{a\Omega} = y\vert_{a\Omega}$)}\\
& = y(a^{n_0+1}h') = y(h a^{n+1}). &&
\end{align*}
Keeping in mind \eqref{e:cas-n}, this shows that $x\vert_{HT_1^{n_0+1}} = y\vert_{HT_1^{n_0+1}}$.
\par
An immediate induction on $m \geq 1$ yields a sequence $(x_m)_{m \geq 1}$ in $\Sigma$ such that 
\begin{equation}
\label{e:x-m-y}
x_m\vert_{HT_1^m} = y\vert_{HT_1^m}
\end{equation}
for all $m \geq 1$.
\par
Let now $F \subset G$ be a finite subset. Then we can find $i,j \in \Z$, with $i \leq j$, such that $F \subset HT_i^j$.
Setting $m \coloneqq j-i+1$, it follows that $a^{-i+1}F \subset HT_1^m$. 
Consider the configuration $y' \coloneqq a^{-i+1}y \in \Sigma'$. Then by using \eqref{e:x-m-y} applied to $y'$, we can find
$x_m' \in \Sigma$ such that $x_m'\vert_{HT_1^m} = y'\vert_{HT_1^m}$.
Then setting $x_m \coloneqq a^{i-1}x_m' \in \Sigma$, we obtain $x_m\vert_{HT_i^j} = y\vert_{HT_i^j}$ so that, in particular,
$x_m\vert_F = y\vert_F$. Since $\Sigma$ is closed and $F$ was arbitrary, this shows that $y \in \Sigma$.
This proves $\widetilde{\Sigma} \subset \Sigma$, and the claim follows.
\end{proof}

The remaining of the proof of the lemma follows step by step the end of the proof of Lemma \ref{l:finite}, with $G$ replaced by $\Omega$
and $\Sigma'$ replaced by $\varphi(\Sigma_\Omega)$.
We thus set $B \coloneqq A^T$, so that $B$ is a finite-dimensional vector space and the map 
$\varphi \colon A^\Omega \to B^H$ defined by \eqref{e:def-phi} is an $H$-equivariant linear isomorphism and uniform homeomorphism.
By virtue of Claim 2, we have that $\Sigma' \coloneqq \varphi(\Sigma_\Omega) \subset B^H$ is a subshift.
Since $H$ is of linear Markov type, and $B$ is finite-dimensional, there exists a finite subset $D_H \subset H$ 
and a subspace $P_H \subset B^{D_H}$ such that $\Sigma' = \Sigma(B^H;D_H,P_H)$.
Then, setting $D \coloneqq D_HT \subset G$ and $P \coloneqq \psi^{-1}(P_H) \subset A^D$, where
$\psi \colon A^D \to B^{D_H}$ is as in \eqref{e:def-phi}, we have that $\Sigma = \Sigma(A^G; D, P)$ is of finite type.
\end{proof}

\begin{proposition}
Let $K$ be a field. Then the class of $K$-linear Markov groups is closed under the operations of taking subgroups, quotients, and extensions by finite or cyclic groups.
\end{proposition}
\begin{proof}
Let $G$ be a group, let $H \subset G$ be a subgroup, and let $A$ be a finite-dimensional vector space over a field $K$.
Given a subshift $\Sigma \subset A^H$ we set
\[
\Sigma^{(G)} \coloneqq \{x \in A^G: (gx)\vert_H \in \Sigma \mbox{ for all } g \in G\} \subset A^G.
\]
Roughly speaking, $\Sigma^{(G)}$ is the set of all configurations in $A^G$ whose restriction to each left coset $c \in G/H$
yields -- modulo the bijection $h \mapsto gh$, induced by an element $g \in c$, which identifies $H$ and $c$ --
and element in $\Sigma$. 
\par
It is easy to see that $\Sigma^{(G)} \subset A^G$ is a linear subshift and that it is of finite type if and only if
$\Sigma$ is (cf.\ \cite[Exercise 1.33]{book}; see also \cite[Lemma 2]{salo}).
We deduce that if $G$ is of linear Markov type, so are all of its subgroups.
\par
Suppose now that $H$ is normal in $G$ and denote by $\pi \colon G \to K \coloneqq G/H$ the canonical quotient homomorphism.
Given a subshift $\Sigma \subset A^K$ we denote by
\[
\Sigma(G) \coloneqq \{x \circ \pi: x \in \Sigma\} \subset A^G.
\]
Roughly speaking, $\Sigma(G)$ is the set of all configurations $x \in A^G$ which are constant on each left coset $c \in G/H$
and such that, if $T \subset G$ is a complete set of representatives of the cosets of $H$ in $G$, then the restriction 
$x\vert_T$ yields -- modulo the bijection $\pi\vert_T \colon T \to K$ -- an element in $\Sigma$.
Assume that $G$ is of linear Markov type.
Once again, it is easy to see that $\Sigma(G) \subset A^G$ is a linear subshift and that it is of finite type if and only if
$\Sigma$ is. We deduce that $K$ is of linear Markov type as well.
\par
The fact that the class of groups of linear Markov type is closed under finite or cyclic extensions follows from Lemma \ref{l:finite} and Lemma \ref{l:cyclic}, respectively.
\end{proof}

It is a well known fact (see, e.g., \cite[Theorem 5.4.12]{robinson})
that a solvable group is polycyclic if and only if it is Noetherian. Similarly, one has that a virtually
solvable group is polycyclic-by-finite if an only if it Noetherian (cf.\ \cite[Lemma 6]{salo}).
From Corollary \ref{c:LMT} and Corollary \ref{c:noetherian-group} we deduce the following (cf.\ \cite[Theorem 5]{salo}):

\begin{corollary}
Let $G$ be a virtually solvable group and let $K$ be a field.
Then the following conditions are equivalent:
\begin{enumerate}[{\rm (a)}]
\item $G$ is of $K$-linear Markov type;
\item $G$ is Noetherian;
\item $G$ is polycyclic-by-finite.
\end{enumerate}
\end{corollary}

\begin{remark}
\label{r:alternative}
As mentioned above, we can directly deduce Corollary \ref{c:LMT} from Lemma \ref{l:finite} and Lemma \ref{l:cyclic}, thus
without using P.~Hall's theorem. For the sake of completeness, we produce here the alternative proof, by induction.
Thus, suppose that $G$ is a polycyclic-by-finite group. Then $G$ admits a subnormal series
$G = G_n \supset G_{n-1} \supset \cdots \supset G_1 \supset G_0 = \{1_G\}$ whose factors are finite or cyclic groups.
We first observe that if $G$ is a trivial group, then it is of $K$-linear Markov type. 
Indeed, let $A$ be a finite-dimensional vector space over a field $K$. Then, setting $D \coloneqq \{1_G\}$ and identifying
$A$ with $A^D$ and $A^G$, we have that the (identity) map
\[
B \mapsto \Sigma(A^G,D,B)
\]
yields a bijection between subspaces $B \subset A$ and subshifts $\Sigma \subset A^G$.
Since every descending sequence of vector subspaces of a finite-dimensional vector space eventually stabilizes, it follows
from Corollary \ref{c:DCC-FT} that all subshifts $A^G$ are of finite type. This proves the base of induction.
A recursive application of Lemma \ref{l:finite} or Lemma \ref{l:cyclic} then
shows that $G_0, G_1, \ldots, G_{n-1}$, and $G_n = G$ are all of $K$-linear Markov type.
\end{remark}

\section{Examples and final remarks}
\subsection{The descending chain condition}
Let $G$ be a group and let $A$ be an infinite-dimensional vector space over a field $K$.
Then $A$ admits a strictly decreasing sequence $(A_n)_{n \in \N}$ of vector subspaces and
the sequence $(\Sigma_n)_{n \in \N}$, where $\Sigma_n\coloneqq A_n^G \subset A^G$, 
is a strictly decreasing sequence of linear subshifts of $A^G$. 
Thus $A^G$ does not satisfy the descending chain condition for linear subshifts. 

\subsection{The closed image property}
\label{s:CIP-examples}
In \cite[Section 5]{cc-TCS} it is shown that if $A$ is an infinite-dimensional vector space and $G$ is any nonperiodic group,
then there exists a linear cellular automaton $\tau \colon A^G \to A^G$ whose image $\tau(A^G)$ is not closed in $A^G$.
This shows that Theorem \ref{t:closed-image} fails to hold in general if the finite-dimensionality of the alphabet $A$ is dropped.
\par
Explicitly, the linear cellular automaton $\tau \colon A^G \to A^G$ we alluded to above can be defined as follows.
Since $A$ is infinite-dimensional, we can find a sequence $(a_i)_{i \in \N}$ of linearly independent vectors in $A$.
Let $E$ denote the vector subspace spanned by the $a_i$'s and let $F$ be a vector subspace such that $A = E \oplus F$.
Let $\psi \colon A \to A$ denote the linear map defined by setting $\psi(a_i) = a_{i+1}$ for all $i \in \N$ and $\psi\vert_F = 0$.
Since $G$ is nonperiodic, there exists an element $g \in G$ of infinite order.
Then the cellular automaton $\tau \colon A^G \to A^G$ with memory set $M = \{1_G,g\} \subset G$ and local defining map
$\mu \colon A^M \to A$ given by
\[
\mu(y) \coloneqq y(g) - \psi(y(1_G)),
\]
for all $y \in A^M$, satisfies that $\tau(A^G)$ is not closed in $A^G$ 
(cf.\ \cite[Lemma 5.2]{cc-TCS}; see also \cite[Example 8.8.3]{book}).

\subsection{Nilpotency for linear cellular automata}
\label{s:examples-nilp-linear}
Let $G$ be a group and let $A$ be a vector space over a field $K$.
Given a linear map $f \colon A \to A$, we denote by $\tau_f \colon A^G \to A^G$ the LCA
with memory set $M \coloneqq \{1_G\}$ and associated local defining map
$\mu_f \coloneqq f  \colon A = A^M \to A$. In other words, $\tau_f = \prod_{g \in G} f$ so that,
in particular, $\tau_f^n(A^G) = \prod_{g \in G} f^n(A)$ for all $n \in \N$. As a consequence,
\begin{equation}
\label{e:omegaf-omegatau}
\Omega(\tau_f) = \bigcap_{n \in \N} \tau_f^n(A^G) = \bigcap_{n \in \N} \prod_{g \in G} f^n(A) =
\prod_{g \in G} \bigcap_{n \in \N}  f^n(A) = \prod_{g \in G} \Omega(f) = \Omega(f)^G.
\end{equation}
Note that $f$ is nilpotent (resp.\ pointwise nilpotent) if and only if $\tau_f$ is nilpotent (resp.\ pointwise nilpotent).
\par
Suppose that $A$ is infinite-dimensional. 
Let $\{e_n: n \in \N\} \subset A$ be an independent subset and set
$A_1 \coloneqq \spa_K\{e_n:n \in \N\}$ and $A_2 \coloneqq A \ominus A_1$.
\par
\begin{enumerate}[{\rm (1)}]
\item
Consider the linear map $f \colon A \to A$ defined by setting $f(e_n) = e_{n+1}$ for all $n \in \N$
and $f(a) = 0$ for all $a \in A_2$. It is then clear that $\Omega(f) = \{0\}$ so that, by \eqref{e:omegaf-omegatau}, 
$\Omega(\tau_f) = \{0\}$. 
However, $\tau_f$ is not pointwise nilpotent (and therefore not nilpotent either).
\item
Consider the linear map $f \colon A \to A$ defined by setting $f(e_0) = 0$, $f(e_n) = e_{n-1}$ for all $n \geq 1$,
and $f(a) = 0$ for all $a \in A_2$. Then $f$ and therefore $\tau_f$ are surjective so that $\tau_f$ 
is not nilpotent, $\Omega(\tau_f) = A^G$. However, $f$ and therefore $\tau_f$ are pointwise nilpotent.
\item
Consider, for each $n \geq 1$, the set $I_n \coloneqq \{0,1,\dots,n\}$ and the map $g_n \colon I_n \to I_n$ 
given by $g_n(k) \coloneqq k - 1$ if $k \geq 1$ and $g_n(0) = 0$. 
Let $X$ be the set obtained by taking disjoint copies of the sets $I_n$, $n \geq 1$, and identifying all copies of 
$0$ in a single point $y_0$ and all copies of $1$ in a single point $y_1 \not= y_0$. 
Then the maps $g_n$ induce a well defined quotient map $g \colon X \to X$. 
Clearly $\Omega(g) = \{y_0,y_1\}$ and $g(\Omega(g)) = \{y_0\}$.
Since $X$ is countable, we can find a bijection $\varphi \colon \N \to X$ such that $\varphi(0) = y_0$ and $\varphi(1) = y_1$.
Setting $h \colon \varphi^{-1} \circ g \circ \varphi \colon \N \to \N$ we thus have $\Omega(h) = \{0,1\}$ and $h(\Omega(h)) = \{0\}$.
Consider the linear map $f \colon A \to A$ defined by setting $f(e_n) \coloneqq e_{h(n)}$ for all $n \in \N$ 
and $f(a) = 0$ for all $a \in A_2$.
Then $\Omega(f) = \spa_K\{e_0,e_1\} = Ke_0 \oplus Ke_1$ while $f(\Omega(f)) = \spa_K\{e_0\} = Ke_0$. 
As a consequence,  $\tau_f(\Omega(\tau_f)) = (Ke_0)^G \subsetneq (Ke_0 \oplus Ke_1)^G = \Omega(\tau_f)$.
\end{enumerate}

\bibliographystyle{siam}

\end{document}